\documentclass{amsart}
\usepackage{tikz,amsthm,amsmath,comment,amsfonts}
\usepackage{xcolor}
\usepackage{textcomp}
\usepackage{stmaryrd}
\usepackage{proofsteps}
\usepackage{bbm}
\usepackage{float}
\usepackage{times}
\usepackage{tcolorbox}
\usepackage[normalem]{ulem}

\usepackage{setspace}
\newtheorem{theorem}{Theorem}
\theoremstyle{plain}

\newtheorem{corollary}[theorem]{Corollary}

\newtheorem{fact}{Fact}
\newtheorem{lemma}[theorem]{Lemma}
\newtheorem{notation}{Notation}
\newtheorem{definition}{Definition}
\newtheorem{preexample}{Example}
\newenvironment{example}%
  {\begin{preexample}\upshape}{\end{preexample}}
\newtheorem{problem}{Problem}
\newcommand{\bprob}{\begin{problem}\begin{normalfont}}
\newcommand{\eprob}{\end{normalfont}\end{problem}}
\newtheorem{proposition}[theorem]{Proposition}

\numberwithin{theorem}{section}

\usepackage{enumitem}
\setlist[enumerate, 1]{
  leftmargin=\parindent,
  parsep=.5\parsep,
  itemsep=0ex
}

\setlist[itemize, 1]{
  leftmargin=\parindent,
  parsep=.5\parsep,
  itemsep=0ex
}

\newlist{tasks}{enumerate}{1}
\setlist[tasks]{
  label=\arabic*\enspace\(\cdot\),
  parsep=.5\parsep,
  itemsep=0ex,
  labelsep=.5em,
  leftmargin=*,
  widest=1
}

\usepackage{listings}
\lstdefinelanguage{Maude}{%
  morekeywords={%
    mod, endm, fmod, endfm, th, endth, fth, endfth, is,
    protecting,
    sort, sorts, subsort,
    var, vars,
    op, ops,
    eq, ceq, rl, crl,
    search
  }%
}
\newcommand{\maudelisting}{%
  \lstset{%
    language=Maude,
    basicstyle=\ttfamily,
    keywordstyle=\bfseries,
    columns=flexible,
    keepspaces=true,
    frame=single,
    escapechar=\#
  }%
}

\usepackage{setspace}
\lstnewenvironment{maude}{%
  \spacing{0.83}
  \maudelisting
}{%
  \endspacing
}

\usepackage{connectives}
\usepackage{xypic}

\newcommand{\st}{\mathrm{st}}
\newcommand{\op}{\mathrm{op}}
\newcommand{\ol}[1]{\overline{#1}}
\newcommand{\aar}{\shortrightarrow}
\newcommand{\heq}{H_=}
\newcommand{\htr}{H_{\aar}}
\newcommand{\co}{\,\colon\;}
\def\sem#1{[\![ #1 ] \!]}

\def\ttbf#1{\textbf{\texttt{#1}}}

\newcommand{\Mod}{\mathit{Mod}}

\newcommand{\lra}{\longrightarrow}

\def\MSA{\mbox{$\mathit{MSA}$}}
\def\POA{\mbox{$\mathit{POA}$}}
\def\RWL{\mbox{$\mathit{RWL}$}}
\newcommand{\hornsen}[3]{\lc{\forall{#1} #2 \implies #3}}

\newcommand{\rew}{%
  \mathrel{\vbox{\offinterlineskip\ialign{%
    \hfil##\hfil\cr
    $\scriptstyle\star$\cr
    \noalign{\kern-0ex}
    $\longrightarrow$\cr
}}}}

\usepackage[hyphens]{url}
\usepackage[colorlinks]{hyperref}
\hypersetup{breaklinks=true}

\begin{document}
\title{Non-deterministic Algebraic Rewriting as Adjunction}
\author{R\u{a}zvan Diaconescu}
\address{Simion Stoilow Institute of Mathematics of the Romanian
  Academy}
\email{Razvan.Diaconescu@ymail.com}

\date{}

%\subjclass[2010]{}

\begin{abstract}
We develop a general model theoretic semantics to rewriting beyond the 
usual confluence and termination assumptions.
This is based on \emph{preordered algebra} which is a model theory
that extends many sorted algebra.
In this framework we characterise rewriting in \emph{arbitrary}
algebras rather than term algebras (called \emph{algebraic rewriting})
as a persistent adjunction and use this result, on the one hand for
proving the soundness and the completeness of an abstract
computational model of rewriting that underlies the non-deterministic
programming with Maude and CafeOBJ, and on the other hand for
developing a compositionality result for algebraic rewriting in the
context of the pushout-based modularisation technique.  
\end{abstract}

\maketitle

\section{Introduction}

Term rewriting is a computational paradigm that plays a major role in
functional programming, theorem proving or formal verification.
We may distinguish between two kinds of term rewriting:
\begin{itemize}

\item Type-theoretic term rewriting that applies to terms of some
  higher-order logical system, such as $\lambda$-calculus.
  A number of functional programming languages and computer systems
  that implement some form of higher-order logic or type theory
  implement this kind of term rewriting, such as Haskell
  \cite{Marlow:Haskell-2010-Report}, ML \cite{ml-book}, Coq
  \cite{coquand-huet88,coq-book}, etc.

\item Algebraic term rewriting that applies to terms in some form of
  model theory, traditionally some variant of universal algebra.
  This includes also particular important contexts such as group
  theory or ring theory where algebraic rewriting plays a crucial role
  in their computational sides.
  In computing science algebraic term rewriting constitutes the
  computational basis of the execution engines of several algebraic
  specification and programming systems such as OBJ \cite{iobj},
  ASF+SDF \cite{asf+sdf}, Elan \cite{elan}, CafeOBJ \cite{caferep},
  Maude \cite{maude-book}, etc.
  While algebraic term rewriting is commonly used as a decision
  procedure for equational logic (by assuming confluence and
  termination) \cite{baader-nipkow1998}, languages and systems such as
  Maude or CafeOBJ also implement algebraic term rewriting in a
  wider non-deterministic sense, its semantic interpretation going
  beyond universal algebra and its associated equational logic.

\end{itemize}
Our work is concerned with the latter kind of term rewriting in its
wider meaning where confluence and even termination are \emph{not}
assumed.
The semantic implication of this is that rather than collapsing the
intermediate states of the rewriting process to the final result (like
in the case of the traditional approach to rewriting as a decision
procedure for equational logic) we consider them as semantic entities
with separate identities. 

There is yet another dimension of generality to our approach to
rewriting.
In our work we consider rewriting in arbitrary algebras (hence the
terminology `algebraic rewriting'), ordinary term rewriting
corresponding just to the particular case when the algebra under
consideration is the term algebra.
Apart of ordinary term rewriting there are several other
interpretations of (general) algebraic rewriting.
One of them concerns the practically important case of rewriting
modulo axioms, the corresponding algebra being the free algebra of an
equational theory.
Another important application of algebraic rewriting is that of
computation with predefined types, when the corresponding algebra is
not user specified but is rather provided by the respective
computational system.

\subsubsection*{Motivation.}
As a programming and specification paradigm, non-deterministic
programming is a powerful one even in a simple form that does not
involve additional fancy features.
But like with any other advanced programming paradigms its
effectiveness depends on methodologies and on a proper understanding
of the paradigm from the side of the users.
Semantics can play a crucial role in all these, but only when it is
simple enough, clean, and it enjoys those mathematical properties that
support several important general techniques such as a good modular
development.
Our work is motivated by this and builds on the logic\footnote{This
  means mathematically defined syntax, semantics and proof theory, as
  put forward in \cite{nel}.} of \emph{preordered algebra}
(abbreviated \POA)
\cite{cafefun,hpoa,iimt,CodescuMRM10,CodescuMRM11} that has been
established as the logical semantics of non-deterministic rewriting in
CafeOBJ \cite{cafefun} and Maude \cite{CodescuMRM10,CodescuMRM11}. 
Here we refine the definition of \POA\ such that it captures more
accurately programming language aspects and develop a series of new
concepts and results regarding non-deterministic rewriting in \POA\ in
order to refine the mathematical foundations of CafeOBJ and Maude
supporting a deeper understanding of how to use them.
Moreover, some novelties in our work may open the door for new
methodologies for non-deterministic rewriting.
We hope this will increase the attractiveness of this paradigm both
for education and software development. 

\subsection{Our main contributions}

\begin{enumerate}

\item A model theoretic semantics to rewriting that has the following
  characteristics (not necessarily ordered by their importance):
  \begin{itemize}

  % \item Mathematical simplicity that avoids the complicated and
  %   sometimes problematic model theoretic structures over
  %   categories as in the `rewriting logic' of \cite{meseguer92}. 

  \item Mathematical effectiveness.
    It is rigorously based on a refinement of the \POA\ of
    \cite{cafefun,hpoa,iimt} which considers two types of sorts
    (rather than a single one) because this fits more accurately the
    applications.
    This upgraded version of \POA\ remains an \emph{institution}
    \cite{ins,iimt}, and this has important consequences, such as the
    availability of a highly developed model theory at the level of
    abstract institutions \cite{iimt} and of a matured specification and
    programming theory that has become a foundational standard when
    defining formal specification and logic-based languages.

  \item The relationship between the model theoretic semantics and the 
    computational model that we establish here is a \emph{direct} one,
    without the mediation of a system of logical deduction (equational
    logic for deterministic term rewriting or rewriting logic in the
    case of non-deterministic term rewriting).  

  \item The models are considered \emph{loose} rather than tight, term
    models being just one example of our theory.
    All interpretations of algebraic rewriting, including those
    mentioned above, can be thus treated in an uniform and unitary
    way. 

  \end{itemize}
  The latter two characteristics represent important novelties with
  respect to previous \POA\ semantics to rewriting. 

\item A categorical adjunction characterisation of algebraic
  rewriting, which can be considered as an axiomatic view on
  rewriting.
  This core result is instrumental for subsequent developments in
  the paper. 

\item An abstract rewriting-based computational model\footnote{In the
    sense of model theory.} that allows for the integration of
  deterministic with non-deterministic rewriting at different layers. 

\item General results about the \emph{soundness} and the
  \emph{completeness} of algebraic rewriting with respect to the
  proposed model theoretic semantics. 
  These properties are best explained as an isomorphism relationship
  between the models of the denotational and of the operational
  semantics.
  An immediate application of these general results is a genuine model
  theoretic argument for the correctness of the executions of Maude
  and CafeOBJ system modules. 

\item A theorem on the compositionality of algebraic rewriting.
  This result provides a model theoretic foundations for the
  structuring of non-deterministic algebraic rewriting through the
  pushout-based technique. 

\end{enumerate}

\subsection{The structure of the paper}

\begin{enumerate}

\item In a section on preliminaries we review some basic notations and 
  concepts about sets, relations and provide a presentation of the
  basic concepts from many-sorted algebra. 

\item We introduce our model theoretic framework for rewriting, the
  (many-sorted) \emph{preordered algebra} in the enhanced form
  mentioned above.

\item We develop the result that shows that the algebraic rewriting
  relations can be defined through an adjunction between the
  category of many sorted algebras and the category of preordered
  algebras satisfying a set of Horn sentences. 

\item We define the abstract operational model for algebraic rewriting
  and we provide sufficient conditions for its soundness and
  completeness  with respect to the denotational model provided by the
  adjunction result.

\item The theorem on the compositionality of algebraic rewriting. 

\end{enumerate}

%%%%%%%%%%%%%%%%%%%%%%%%%%%%%%%%%%%%%%%%%%%%%%%%%%
\section{Preliminaries}

\subsection{Relations}
A \emph{binary relation} is a subset $R \subseteq A \times B$, where
$A$ and $B$ are sets.
When $A = B$ we say that $R$ is a binary relation on $A$.
Sometimes $(a,b)\in R$ may be denoted as $a \ R \ b$. 
A binary relation $R$ on a set $A$ is
\begin{itemize}

\item \emph{reflexive}, when for each $a\in A$, $(a,a) \in R$,

% \item \emph{anti-reflexive}, when for each $a\in A$, $(a,a) \not\in
%   R$, 

\item \emph{symmetric}, when for each $a, b \in A$, $(a,b) \in R$
  implies $(b,a) \in R$,

\item \emph{anti-symmetric}, when for each $a, b \in A$, $(a,b),(b,a)
  \in R$ implies $a=b$, 

\item \emph{transitive}, when for each $a,b,c \in A$, $(a,b), (b,c)
  \in R$ implies $(a,c)\in R$,  

\item \emph{preorder}, when it is reflexive and transitive,

\item \emph{partial order}, when it is anti-symmetric preorder,

\item \emph{equivalence}, when it is symmetric and preorder.
  
\end{itemize}
Given an equivalence relation $\sim$ on a set $A$, for each $a\in A$
by $a/_\sim$ we denote the \emph{equivalence class of $a$} which is
the set $\{ b \in A \mid a \sim b \}$.
Let $A/_\sim = \{ a/_\sim \mid a \in A \}$ be the set of all
\emph{equivalence classes} in $A$ determined by $\sim$.

Given a relation $R$ let $R^*$ be its \emph{reflexive-transitive
closure}, which is the least reflexive and transitive relation
containing $R$.
Then $R^* = \bigcup_{n\in \omega} R^n$ where
$R^n = \{ (a_0, a_n) \mid \exists (a_i, a_{i+1}) \in R, \
0 \leqslant i \leqslant n - 1 \}$.
For a binary relation $\lra$ by $\rew$ we mean $\lra^\star$. 

\subsection{Many-sorted Algebra}

Many-sorted algebra (\MSA) has traditionally played a major role in
the semantics of computation.
For instance \MSA\ is the core formalism in the area of algebraic
specification.
Since OBJ \cite{iobj}, due to algebraic term rewriting, \MSA\ has been
turned into a functional programming paradigm (called equational
programming). 
Since the model theoretic framework of our work is built on top of
\MSA, in what follows we provide an overview of the main \MSA\
concepts.  

\subsubsection{Signatures}
We let $S^*$ 
%\index{sind}{$S^*$} 
denote the set of all finite sequences of elements from $S$, 
with $[]$ the empty sequence. 
A(n $S$-\emph{sorted}) \emph{signature} $(S,F)$ is an
$S^* \times S$-indexed set
$F = \{ F_{w \aar s} \mid w \in S^*,\ s \in S \}$
%\index{sind}{$F_{w \aars}$}
of \emph{operation symbols}.
The sets $F_{w \aar s}$ in the definition above stand for 
the sets of symbols with arity $w$ and sort $s$.
We may denote $F_{[]\aar s}$ simply as $F_{\aar s}$.
Note that this definition permits 
{\em overloading}\index{overloading of operation symbols}, 
in that the sets $F_{w \aar s}$ need {\em not\/} be disjoint.
When this leads to parsing ambiguities, the qualification of the
operation symbols by their rank is a good solution.\footnote{This is
  also the solution adopted in the implementation of algebraic
  specification languages.}
Then $\sigma \in F_{w \aar s}$ will be denoted as $\sigma \co \! w \aar s$. 

\subsubsection{Terms}
An \emph{$(S,F)$-term} $t$ of sort $s\in S$, is an
expression of the form $\sigma(t_1,\dots,t_n)$, where
$\sigma \in F_{s_1 \dots s_n \aar s}$ and $t_1, \dots, t_n$ are
$(S,F)$-terms of sorts $s_1, \dots, s_n$, respectively. 
The set of the $(S,F)$-terms of sorts $s$ is denoted by
$(T_{(S,F)})_s$ while $T_{(S,F)}$ denotes the set of all
$(S,F)$-terms.

\subsubsection{Variables}
Let $(S,F)$ be a signature.
A \emph{variable for $(S,F)$}\index{variable (for signature)}
is a pair $(x,s)$ where $x$ is the \emph{name of the variable} and 
$s\in S$ is the \emph{sort of the variable}, such that $(x,s) \not\in
F_{\aar s}$. 
The sort is an essential qualification for a variable.
However, when this is clear, we may simply refer to a variable by its
name only.
For example, if $X$ is a set of variables for $(S,F)$, 
then $(x,s) \in X$ may be denoted $x\in X$.
For this to make sense, but also in order to avoid other kinds of clashes, 
we make the basic assumption valid all over our material, that when 
considering  sets of variables, 
\emph{any two different variables have different names}.

For any signature $(S,F)$ and any set $X$ of variables for $(S,F)$, 
the signature $(S,F\cup X)$ denotes the extension of $(S,F)$ with $X$ as 
new constants that respects the sorts of the variables. 
This means $(F \cup X)_{w \aar s} = F_{w \aar s}$ when $w$ is not empty and 
$(F \cup X)_{\aar s} = F_{\aar s}\cup \{ (x,s) \mid (x,s) \in X
\}$.\footnote{Note that this union is always disjoint because $X$ are
  \emph{new} constants.}
% \sout{Also, when $\Sigma$ means $(S,F)$ then $\Sigma + X$ means
%   $(S,F\cup X)$.}

\subsubsection{Substitutions}
Given sets $X$ and $Y$ of variables for a signature $(S,F)$,
an \emph{$(S,F)$-substitution} $\theta$ from $X$ to $Y$ 
is a function $\theta \colon  X \to T_{(S,F \cup Y)}$ that respects
the 
sorts, i.e. if $x$ has sort $s$ then 
$\theta(x) \in (T_{(S,F \cup Y)})_s$. 
The existence of substitutions from $X$ to $Y$ requires that 
whenever there is a variable in $X$ of sort $s$ then  
$(T_{(S,F \cup Y)})_s$ is non-empty.
In general, this condition can be met if we assume that the 
signatures contain at least one constant for each sort. 
Any substitution $\theta \colon  X \to T_{(S,F \cup Y)}$
extends to a function 
$\theta^\sharp \colon   T_{(S,F \cup X)} \to T_{(S,F \cup Y)}$ defined by 
\[
 \theta^\sharp (t) = \left\lbrace
\begin{array}{ll}
 \theta(x) &  \mbox{when \ }  t = x \mbox{ \ for \ } x\in X \\
  \sigma(\theta^\sharp(t_1),\dots,\theta^\sharp(t_n)) & 
  \mbox{when \ } t = \sigma(t_1,\dots,t_n) \mbox{ \ with \ } \sigma \in F.
\end{array}\right.
\]
When there is no danger of notational confusion we may omit `$\sharp$'
from the notation and write simply $\theta(t)$ instead of 
$\theta^\sharp (t)$.

\subsubsection{Sentences}
Given a signature $(S,F)$, an \emph{atomic equation} is a pair
$(t,t')$ of $(S,F)$-terms of the same sort. 
Often we write atomic equations as $t=t'$ rather than $(t,t')$. 
The set of the \emph{$(S,F)$-sentences} is the least set such that: 
\begin{itemize}
 
\item Each atomic $(S,F)$-equation is an $(S,F)$-sentence.

\item If $\rho_1$ and $\rho_2$ are $(S,F)$-sentences then
$\lc{\rho_1 \and \rho_2}$ (\emph{conjunction}),  
$\lc{\rho_1 \or \rho_2}$ (\emph{disjunction}),
$\lc{\rho_1 \implies \rho_2}$ (\emph{implication}) and 
$\lc{\not \rho_1}$ (\emph{negation}) 
are also $(S,F)$-sentences.

\item If $X$ is a set of variables for $(S,F)$, 
then $\lc{\forall{X} \rho}$ and $\lc{\exists{X} \rho}$ are
$(S,F)$-sentences whenever $\rho$ is an $(S,F \cup X)$-sentence. 
 
\end{itemize}

The sentences that do not involve any quantifications are called
\emph{quantifier-free sentences}.\index{quantifier-free sentences} 

A \emph{conditional equation}\index{conditional equation} is a
sentence of the form
\[
  \lc{\forall{X} H \implies C}
\]
where $H$ is a finite conjunction of (atomic) equations and $C$ is a
single (atomic) equation.  

When $H$ is empty the respective conditional equation
is usually written simply as $\lc{\forall{X} C}$ and is called 
\emph{unconditional equation}.\index{unconditional equation}
In this paper 'conditional equations' will be called simply
'equations' and when the equations are unconditional we will
explicitly say 'unconditional equations'. 

\subsubsection{Algebras}

Given a set of sort symbols $S$, an $S$-\emph{indexed} (or \emph{sorted})
\emph{set} $A$ is a family $( A_s )_{s \in S}$ of sets indexed by
the elements of $S$; in this context, $a \in A$ means that $a \in A_s$
for some $s \in S$. 
Given an $S$-indexed set $A$ and $w = s_1...s_n \in S^*$, we let
$A_w = A_{s_1} \times \cdots \times A_{s_n}$; in particular, we let
$A_{[]} = \{\star\}$, some one point set.

Given a signature $(S,F)$, a $(S,F)$-\emph{algebra}\index{algebra} 
$A$ consists of
\begin{itemize}

\item an $S$-indexed set $A$ (the set $A_s$ is called the
\emph{carrier} of $A$ of sort $s$), and

\item a function $A_{\sigma : w \aar s} \colon  A_w \to A_s$ 
for each $\sigma \in F_{w \aar s}$.

\end{itemize}
%The class of all $(S,F)$-algebras is denoted by $\Alg(S,F)$.
When there is no danger of ambiguity (because of overloading of $\sigma$)
we may simplify the notation $A_{\sigma : w \aar s}$ to $A_\sigma$.  
If $\sigma \in F_{\aar s}$ then $A_\sigma$ determines an element 
in $A_s$ which may also be denoted $A_\sigma$.

%\sout{Let $(S,F)$ be any signature.} 
Any $(S,F)$-term $t = \sigma(t_1, \dots, t_n)$, where $\sigma$ is an
operation symbol of $(S,F)$, gets interpreted as an element $A_t \in
A_s$ in a $(S,F)$-algebra $A$ defined by  
\[
A_t = A_\sigma (A_{t_1}, \dots, A_{t_n}). 
\]

\subsubsection{Term algebras}
Given a signature $(S,F)$, the \emph{$(S,F)$-term algebra}, denoted
$0_{S,F)}$, interprets any sort symbol $s \in S$ as the set of the
$(S,F)$-terms of sort $s$, and each operation symbol $\sigma \co \! w
\aar s$ as
\[
(0_{(S,F)})_\sigma (t_1,\ldots, t_n) = \sigma(t_1,\ldots, t_n).
\]

\subsubsection{Homomorphisms}
An $S$-\emph{indexed} (or \emph{sorted}) \emph{function}
$f\colon  A \to B$ is a family $\{ f_s\colon  A_s \to B_s \mid s\in S \}$.
Also, for an $S$-sorted function $f\colon  A \to B$, we let
$f_w \colon  A_w \to B_w$ denote the function product
mapping a tuple of elements $(a_1, \dots, a_n)$ to the tuple
$( f_{s_1}(a_1), \dots, f_{s_n}(a_n) )$. 

Given $(S,F)$-algebras $A$ and $B$, an $(S,F)$-\emph{homomorphism}
from $A$ to $B$ is an $S$-indexed function  $h\colon  A \to B$ such
that  
\[
h_s(A_\sigma (a)) = B_\sigma(h_w (a))
\]
for each $\sigma \in F_{w \aar s}$ and $a \in A_w$.
When there is no danger of confusion we may simply write $h(a)$ 
instead of $h_s (a)$.

Homomorphisms preserve the interpretations of terms, which is a useful
technical property:

\begin{lemma}\label{term-pres-lem}
Let $h \co A \to B$ be an $(S,F)$-algebra homomorphism and let $t$ be
any $(S,F)$-term.
Then $h(A_t) = B_t$.
\end{lemma}

Given $(S,F)$-homomorphisms $h\colon  A \to B$ and $g \colon  B \to
C$, their  composition $h;g$ is the algebra homomorphism $A \to C$
defined by $( h;g )_s = g_s \circ h_s$  for each sort symbol 
$s \in S$.\footnote{This means $(h;g)(a) = g(h(a))$ for 
each $a \in A_s$.}

\subsubsection{Isomorphisms}
An $(S,F)$-homomorphism $h\colon  A \to B$ is a $(S,F)$-\emph{isomorphism}
\index{isomorphism of algebras}
when there exists another homomorphism $h^{-1} \colon  B \to A$ such that 
$h;h^{-1} = 1_A$ and $h^{-1};h = 1_B$, where by $1_A \colon  A \to A$ and 
$1_B \colon  B \to B$ we denote the `identity' homomorphisms that map 
each element to itself.

\begin{fact} \label{iso-bij-fact}
A $(S,F)$-homomorphism $h\colon  A \to B$ is isomorphism if and 
only if each function $h_s\colon  A_s \to B_s$ is bijective (i.e.,
one-to-one and onto, in an older terminology).
\end{fact}

\subsubsection{Initial algebras}
Given any class $\mathcal{C}$ of $(S,F)$-algebras, an algebra $A$ is 
\emph{initial}\index{initial algebra} for $\mathcal{C}$ when 
$A \in \mathcal{C}$ and for each algebra $B \in \mathcal{C}$
there exists an unique homomorphism $f \co A \to B$.

\begin{proposition}
If $A$ and $A'$ are both initial algebras for $\mathcal{C}$, then 
there exists an isomorphism $h \co A \to A'$. 
\end{proposition}

The following is well known and also easy to establish.

\begin{proposition}
The term algebra $0_{(S,F)}$ is initial in the class of all
$(S,F)$-algebras.
\end{proposition}

\subsubsection{Satisfaction}
The satisfaction between $(S,F)$-algebras and $(S,F)$-sentences, 
denoted $\models_{(S,F)}$ (or simply by $\models$ when there is no
danger of confusion), is defined inductively on the structure of the
sentences as follows. 
This process can be regarded as an evaluation of the sentences  to one
of the truth values \texttt{true} or \texttt{false} and which is
contingent on a given model/algebra. 

Given a fixed arbitrary signature $(S,F)$ and a $(S,F)$-algebra $A$, 
\begin{itemize} 

\item $A \models t=t'$ if and only if  $A_t = A_{t'}$ for atomic
  equations, 

\item $\lc{A \models \rho_1 \and \rho_2}$ if and only if 
$A \models \rho_1$ and $A \models \rho_2$, 

\item $\lc{A \models \rho_1 \or \rho_2}$ if and only if
$A \models \rho_1$ or $A \models \rho_2$, 

\item $\lc{A \models \rho_1 \implies\rho_2}$ if and only if
$A \not\models \rho_1$ or $A \models \rho_2$, 

\item $\lc{A \models \not \rho}$ if and only if
$A \not\models \rho$, 

% \end{itemize}
% for all $(S,F)$-sentences $\rho_1$ and $\rho_2$, and 
% \begin{itemize}
 
\item for any set of variables $X$ for the signature $(S,F)$, and
for any $(S,F \cup X)$-sentence $\rho$, 
$\lc{A \models_{(S,F)} \forall{X}\rho}$ if and only if 
$A' \models_{(S,F \cup X)} \rho$ for each $(S,F \cup X)$-algebra $A'$ 
such that $A'_s = A_s$ for each $s\in S$ and 
$A'_\sigma = A_\sigma$ for each operation symbol $\sigma$ of $F$. 

\item $\lc{A \models \exists{X} \rho}$ if and only if 
$A \not\models \lc{\forall{X} \not \rho}$. 

\end{itemize}
When $A \models \rho$ we say that $A$ \emph{satisfies} $\rho$
or that $\rho$ \emph{holds} in $A$.

\subsubsection{Signature morphisms}
Given two \MSA\ signature $(S,F)$ and $(S',F')$, a \emph{signature
  morphism} $\varphi \co (S,F) \to (S',F')$ consists of a function
$\varphi^\st \co S \to S'$ and a family of functions 
  $\varphi^\op = 
  \{ \varphi^\op_{w\aar s} \co F_{w\aar s} \to F'_{\varphi^\st (w) \aar
    \varphi^\st (s)} \mid w\in S^*, s\in S \}$.
Signature morphisms compose component-wise; we skip the
straightforward technical details here.
The signature morphisms have all compositionality properties of the
functions. 

\subsubsection{Sentence translations}
Any \MSA\ signature morphism $\varphi \co (S,F) \to (S',F')$ induces a
translation function from the set of the $(S,F)$-sentences to the set
of the $(S',F')$-sentences.
In brief, these translations just rename the sort and the function
symbols according to the respective signature morphism. 
They can be formally defined by induction on the structure of the
sentences. 
The details of this can be read from the literature
(eg. \cite{iimt}). 
Let us denote by $\varphi \rho$ the translation of an $(S,F)$-sentence
$\rho$ along a signature morphism $\varphi$. 

\subsubsection{Model reducts}
For each signature morphism $\varphi \co (S,F) \to (S',F')$, the
\emph{$\varphi$-reduct} $\varphi A'$ of an $(S',F')$-algebra $A'$ is an
$(S,F)$-algebra defined by $(\varphi A')_x = A'_{\varphi x}$ for each
sort or function symbol $x$ from $(S,F)$.
Conversely, $A'$ is called the \emph{$\varphi$-expansion} of $A$.
These concepts extend to model homomorphisms. 
Given an $(S',F')$-homomorphism $h' \co A' \to B'$, its
$\varphi$-reduct is the model homomorphism $h \co \varphi A' \to
\varphi B'$, denoted $\varphi h'$, defined for each sort $s \in S$ by
$h_s = h'_{\varphi^\st s}$. 

\subsubsection{Assignment evaluations}
Given a signature $(S,F)$, a set $X$ of variables for $(S,F)$, an
$(S,F)$-algebra $B$, and a mapping $\theta \co X \to B$, for any $(S,F
\cup X)$-term $t$ by $\theta t$ we designate $B'_t$ where $B'$ is the
$(S,F \cup X)$-expansion of $B$ such that $B'_x = \theta x$ for each
$x \in X$. 
When $B$ is the term algebra $0_{(S, F \cup Y)}$ then $\theta$ is a
substitution $X \to Y$ and $\theta t$ is precisely the instance of $t$
by $\theta$ as defined above.
Thus, in this way the mappings $\theta \co X \to B$ can be regarded as
a generalisation of the concept of substitution. 

\subsubsection{The Satisfaction Lemma}
It says that:

\begin{theorem}\label{sat-lem-msa}
For each \MSA\ signature morphism $\varphi \co (S,F) \to
(S',F')$, for each $(S,F)$-sentence $\rho$ and each $(S',F')$-algebra
$A'$
\[
A' \models_{(S',F')} \varphi \rho
\mbox{ \ if and only if \ }
\varphi A' \models_{(S,F)} \rho.
\]
\end{theorem}
Proofs of this result can be found in several places in the
literature, eg. \cite{ins,iimt}. 

The capture of the \MSA\ signatures and their morphisms as a category,
of the collection of sets of sentences and their translations  as a
functor, of the collection of categories of models and their reducts
as another functor, together with the Satisfaction Lemma, is what
makes \MSA\ an \emph{institution}.  
As shown in the literature this has vast theoretical and practical
consequences which we do not discuss here.
In this work we do not involve any institution theory, however we want 
the reader to be aware in general of the importance of having model
theoretic frameworks captured as institutions, and in particular of
the importance of having the model-theoretic framework of our work
captured as an institution.  

%%%%%%%%%%%%%%%%%%%%%%%%%%%%%%%%%%%%%%%%%%%%%%%%%%%
%%%%%%%%%%%%%%%%%%%%%%%%%%%%%%%%%%%%%%%%%%%%%%%%%%%
\section{Many-sorted Preordered Algebra}

The structure of `preordered algebra', has already been present in
works regarding model theoretic semantics of transitions
(e.g. \cite{cafefun,hpoa,CodescuMRM10,CodescuMRM11})
and has also been used frequently as one of the benchmark examples for
institutional model theoretic developments
\cite{edins,upins,gaina-petria2010,codescu-gaina2008,gaina-popescu2004}.  
Much of the model theory of \POA\ can be found in \cite{iimt}.
Although it has its own specific computing science meaning, model
theoretically \POA\ is very much an extension of \MSA.
Here we upgrade and refine the established definitions of \POA\ in
order to make it more usable in the applications.
The most important upgrade -- that triggers also most of the other
upgrades -- is that we will make an explicit
distinction between two kind of sorts, one for data types (like in
\MSA) and another kind for transitions. 
Only the latter sorts are interpreted as preorders.
Although all developments of our paper can be done within the context
of the simpler form of \POA, only with the upgraded \POA\ we will be
able to match properly the mathematical theory to the specification
examples.   

\subsection{\POA\  definitions}
They follow the structure of the \MSA\ definitions and also include
them. 
Therefore we will provide explicitly only the definitions of the parts
that are specific to \POA.  

\begin{definition}[\emph{POA} signatures]
A \emph{\emph{POA} signature} is a triple $(D,S,F)$ where $D$, $S$ are
sets such that $D \cap S = \emptyset$ and $(D \cup S, F)$ is an
\emph{MSA} signature.
The elements of $D$ are called \emph{data sorts} while the
elements of $S$ are called \emph{system sorts}. 
\end{definition}
% In the case of the signature of \texttt{BUBBLE-SORT}, $D = \{
% \texttt{Nat} \}$, $S = \{ \texttt{List\{Nat\}} \}$, and $F$ consists
% of the general operations on lists such as concatenation, \texttt{nil}
% (the empty list), etc. and the operations on naturals that come with
% the module \texttt{NAT}.

The \POA\ concepts of term, variable, substitution, are all just the
corresponding \MSA\ concepts obtained by regarding the \POA\
signatures $(D,S,F)$ as \MSA\ signatures $(D \cup S, F)$. 

\begin{definition}[\POA\ sentences]\label{poa-sent}
Given a \POA\ signature $(D,S,F)$, the \emph{$(D,S,F)$-sentences} are
defined like \emph{MSA} $(D\cup S,F)$-sentences with one important
difference: instead of one kind of atoms like in \emph{MSA}, in \POA\
we have two kinds of atoms: 
\begin{itemize}
 
\item \emph{MSA} (atomic) \emph{equations} $t = t'$, where $t,t'$ are
  $(D\cup S,F)$-terms of the same sort, and 

\item (atomic) \emph{transitions} $\lc{t \trans t'}$ where $t, t'$ are
  $(D\cup S,F)$-terms of the same \emph{system} sort.\index{transition} 

\end{itemize}
A  \emph{Horn $(D,S,F)$-sentence}\index{Horn sentence in \POA} is a
sentence of the form $\hornsen{X}{H}{C}$ where $H$ is a finite  
conjunction of atoms and $C$ is a single atom.
A \emph{$(D,S,F)$-equation} is just a $(D \cup S,F)$-equation. 
A \emph{$(D,S,F)$-transition} is any Horn $(D,S,F)$-sentence
$\hornsen{X}{H}{(\lc{t \trans t'})}$. 
\end{definition}

% With respect to how Maude implements \POA\ signatures and sentences,
% we mention the following important aspects:
% \begin{itemize}

% \item Maude does not provide any syntax for distinguishing between
%   data and system sorts.
%   This distinction is done rather implicitly based on the sentences.
%   When there is a sentence $\hornsen{X}{H}{(t \trans t')}$ then the sort
%   of $t, t'$ is considered a system sort.
%   The sorts that are not identified as being system sorts in this way
%   are considered data sorts.

% \item Maude does not provide a Boolean-valued operation associated to
%   transitions like it does for the equality with \texttt{==}.
%   Hence the only way to use transitions in the conditions is the
%   standard way that corresponds literally to Definition
%   \ref{poa-sent}.
%   In this case one cannot encode conditions as Boolean terms.
  
% \item Maude does support only equations and transitions, it does not
%   support Horn sentences of the form $\hornsen{X}{H}{(t=t')}$ when
%   $H$ contains atomic transitions.

% \end{itemize}

In some situations it is useful to be able to separate in the
conditions $H$ the equations from the transitions.

\begin{notation}
  Given a finite conjunction $H$ of atomic sentences we let
  \begin{itemize}

  \item 
    $\heq$ be the \emph{set} of the equations occurring in $H$, and

  \item $\htr$ be the \emph{set} of the transitions occurring in $H$.

  \end{itemize}
\end{notation}

\begin{definition}[\POA\ models]
Given a \POA\ signature $(D,S,F)$, a \emph{$(D,S,F)$-(preordered)
  algebra} $(A,\leqslant)$ consists of
\begin{itemize}

\item an \emph{MSA} $(D\cup S, F)$-algebra $A$, and

\item a family $\leqslant \ = \ (\leqslant_s)_{s\in S}$ such that each
  $(A_s, \leqslant_s)$ is a preordered set
  
\end{itemize}
such that for each $\sigma \in F_{w \aar s}$ if $s \in S$ then
$A_\sigma$ is monotone.
\end{definition}

\begin{definition}[\POA\ model homomorphism]
Given $(D,S,F)$-algebras $(A,\leqslant)$ and $(B,\leqslant)$, a \emph{\POA\
  homomorphism} $h \co (A,\leqslant) \to (B,\leqslant)$ is an \MSA\ homomorphism
$h \co A \to B$ that is monotone on the system sorts, ie. $a \leqslant a'$
implies $ha \leqslant ha'$. 
\end{definition}

The \POA\ concepts of isomorphism and initiality are like the
respective \MSA\ concepts.
In fact both concepts are category-theoretic concepts, so their
definition is independent of the actual model theoretic framework.
However there is an important difference between \MSA\ and \POA\ with
respect to model isomorphisms: in \POA\ Fact \ref{iso-bij-fact} does
\emph{not} hold!
The concept of \POA\ model isomorphism is strictly stronger than that
of bijective homomorphism, in other words a bijective \POA\
homomorphism is not necessarily a \POA\ isomorphism.

\begin{example}
Let us consider the signature with one sort and two constants $a$ and
$b$, and on the one hand the \POA\ algebra $A$ consisting of two
elements, $A_a$ and $A_b$ and an empty preorder relation, and on the
other hand the \POA\ algebra $B$ consisting of two elements $B_a$ and
$B_b$ such that $B_a \leqslant B_b$.
Then the unique homomorphism $h \co A \to B$ is bijective but its
inverse as a function is \emph{not} a \POA\ homomorphism because it is
not monotone. 
\end{example}

\begin{definition}[\POA\ satisfaction]
Given a \POA\ signature $(D,S,F)$, the satisfaction between
the $(D,S,F)$-algebras and $(D,S,F)$-sentences is obtained by
extending the \emph{MSA} satisfaction between $(D \cup S, F)$-algebras
and $(D \cup S,F)$-sentences with the satisfaction of atomic
transitions:   
\[
 (A,\leqslant)\models_{(D,S,F)} \lc{t \trans t'} 
\mbox{ \ if and only if \ } A_t \leqslant A_{t'}.
\]
and by following the same inductive process on the structure of the
sentences like in \MSA. 
\end{definition}

We end this section with a couple of examples.

\begin{example}\label{bubblesort-ex}
The bubble sort algorithm for sorting lists of natural numbers admits
a very compact and clear coding in \POA\ as follows.\footnote{We use
  the Maude notations that are close to the common mathematical
  notations as follows. A module definition starts with its name (we
  use the keyword \ttbf{mod} without respect of the semantic nature of
  the respective module), then  we have data type import declarations,
  then sort symbols declarations (keyword \ttbf{sort}), then operation
  symbols declarations (keyword \ttbf{op}) that follow the usual
  functional notation from mathematics that was introduced by Euler,
  then variables declarations (keyword \ttbf{var}) in which the sort
  of each variable is given. These give the specification of the
  signatures. By default convention a sort is a system sort if and
  only if there exist a transition with terms of that sort. The
  equations / transitions are specified by the keywords \ttbf{eq} /
  \ttbf{rl} (or \ttbf{ceq} / \ttbf{crl} in the conditional
  case). \texttt{=>} is used for transitions. The conditions are given
  as Boolean terms after the keyword \ttbf{if}.}  
\begin{maude}
mod BUBBLE-SORT is
  protecting LIST{Nat} .
  vars m n : Nat .
  crl m n => n m if n < m .
endm
\end{maude}
In this specification there is an import of data using the keyword
\ttbf{protecting}.
The data is the lists of the natural numbers.
This has two types:
\begin{enumerate}

\item The natural numbers \texttt{NAT}.
  There are two ways to consider this, either as user defined or as
  predefined.
  In the former approach \texttt{NAT} is an equational specification
  of the natural numbers as can be found in many places in the
  algebraic specification literature.
  Although this is a nice example to illustrate concepts from the
  theory of the initial data type specification, the real world way is
  to consider \texttt{NAT} as a predefined type.
  Practically this means that \texttt{NAT} is already made available
  by the system.
  In the case of Maude and CafeOBJ it actually comes from the
  underlying implementation language, which is C ultimately.
  In terms of the mathematical semantics this means that we have a
  signature that consists of one sort \texttt{Nat} and some usual
  operations on the natural numbers.\footnote{We take here a
    minimalist approach that does not consider order sorted
    structures.}
  For the purpose of this example we only need the predicate ``less
  than'' which is encoded as a Boolean valued operation. In Maude
  notation we can write this as
\begin{maude}
  sort Nat .
  op _<_ : Nat Nat -> Bool .
\end{maude}
Next we consider a standard interpretation of the natural numbers as an
\MSA\ algebra, where \texttt{Nat} is interpreted as $\omega$ (the set
of the natural numbers) and \texttt{\_<\_} is interpreted as the
``less then'' relation on $\omega$.

\item The list of the natural numbers \texttt{LIST\{Nat\}}.
  In Maude and CafeOBJ this is also a predefined module but in a
  different way than \texttt{NAT} is.
  It does not come from C, but it is rather defined as an equational
  specification. It is predefined in the sense that is already
  available in the system.
  Here we give it the simplest possible treatment:
\begin{maude}
  sort List{Nat} .
  op _ : Nat -> List{Nat} .
  op nil : -> List{Nat} .
  op __ : List{Nat} List{Nat} -> List{Nat} [assoc] .
\end{maude}
The first operation provides the generators for the lists, which are
the natural numbers regarded as lists of size 1.
In most algebraic specification languages this is handled more
elegantly by declaring \texttt{Nat} as a \emph{subsort} of
\texttt{List\{Nat\}}.
However this requires \emph{order sorted algebra} \cite{osa-survey}
which is a refinement of \MSA\ that is outside the scope of our
discussion here. 
The constant \texttt{nil} represents the empty list and the binary
operation represent the concatenation (written in mixfix syntax).
The word \texttt{assoc} abbreviates the equation specifying the
associativity of the binary operation. 
The considered \MSA\ algebra interprets as the set of lists of natural
numbers (denoted $\omega^*$) and \texttt{\_\_} as lists
concatenation. 
  
\end{enumerate}
These are \emph{common to all models of} \texttt{BUBBLE-SORT}; this is
the meaning of \ttbf{protecting}.
From now on we can have various \POA\ models that satisfy the
transition specified in the module.
But before presenting some of them it is important to clarify the $D$
and the $S$ from the definition of the \POA\ signatures.
$D = \{ \texttt{Nat}, \texttt{Bool} \}$ and $S = \{
\texttt{List\{Nat\}} \}$ which means   that the \POA\ models do not
consider a preorder on $\omega$, the   interpretation of
\texttt{Nat}.
$F$ consists of all operations of \texttt{NAT}, whatever they may be,
\texttt{\_<\_} and the two \texttt{LIST} operations. 
What happens is that we start with a \MSA\ algebras, the set of sorts
being $D \cup S$, and then we consider $(D,S,F)$-algebras by adding
appropriate preorders.
Let us describe some of the $(D,S,F)$-algebras thus obtained.
\begin{enumerate}

\item The preorder on $\omega^*$ is the total preorder, i.e. $l
  \leqslant l'$ for all $l,l' \in \omega^*$.

\item The preorder on $\omega^*$ is given by the reflexive-transitive
  closure of the relation
  \[
  (l \ m \ n \ l') \lra (l \ n \ m \ l')
  \]
  where $l,l' \in \omega^*$, $m,n \in \omega$.

\item The preorder $\leqslant$ on $\omega^*$ is given by 
  \[
    l \leqslant   l' \mbox{ \  if and only if \ } \mathrm{inv}(l) \leq
      \mathrm{inv}(l')
    \]
where $\mathrm{inv}(l)$ represents the number of the inversions in
$l$. 

\end{enumerate}
If we match the transition specified by \texttt{BUBBLE-SORT} to the
general definition $\hornsen{X}{H}{t \trans t'}$ we have that
$X$ is $\{ m, n \}$, $H$ is the equation $(\texttt{n} < \texttt{m}) =
\texttt{true}$, $t$ is $(\texttt{m n})$ and $t'$ is
$(\texttt{n m})$. 
The first two models satisfy the transition of \texttt{BUBBLE-SORT}
while the third model does not satisfy it.\footnote{Since for instance
  $(2 \ 1) \not\leqslant (1 \ 2)$ which means that for $m=2, n=1$ the
  condition of the transition holds while the transition itself does
  not hold.} 
\end{example}

\begin{example}\label{sol-ex}
  Let us consider the following simple specification:
\begin{maude}
mod SOLUTIONS is 
  sort s . 
  op f : s -> s .
  var x : s .
  rl f(x) => x .
endm
\end{maude}
From the wide class of the \POA\ models for the signature of this
specification we consider the following ones:
\begin{enumerate}

\item The model $(A,\leqslant_A)$: $A_s = \mathbb{Z}$ (the set of the
  integers), $A_f (n) = n^2$, $m \leqslant_A n$ if and only if
  $n^{2^k} = m$ for some $k \in \omega$.

\item The model $(B,\leqslant_B)$: $B_s = \mathbb{Z}$, $B_f (n) =
  n^2$, $m \leqslant_B n$ if and only if $|n| \leqslant |m|$.

\item The model $(C,\leqslant_C)$: $C_s = \mathbb{Z}$, $C_f (n) =
  n^2$, $m \leqslant_C n$ if and only if $m$ divides $n$.  
  
\end{enumerate}
While the first two models do satisfy the transition of
\texttt{SOLUTIONS} the third one does not.\footnote{Since for instance
$C_f (3) = 9 \not\leqslant_C 3$.} 
\end{example}  

\subsection{\POA\ as institution}
\label{poa-institution}

\begin{definition}
A \POA\ signature morphism $\varphi \co (D,S,F) \to (D',S',F')$ is
just an \MSA\ signature morphism $\varphi \co (D \cup S,F) \to (D'
\cup S', F')$ such that $\varphi^\st D \subseteq D'$ and $\varphi^\st
S \subseteq S'$. 
Moreover
\begin{itemize}

\item Each \POA\ $(D,S,F)$-sentence $\rho$ gets translated to a
  $(D',S',F')$-sentence $\varphi \rho$ by following the \MSA\
  translation of sentences that correspond to $\varphi$ regarded as an
  \MSA\ signature morphism $(D\cup S, F) \to (D' \cup S',F')$.
  The atomic transitions are translated in the same way as the atomic
  equations are translated.

\item Each \POA\ $(D',S',F')$-algebra $(A',\leqslant)$ gets reduced to
  the $(D,S,F)$-algebra $\varphi(A',\leqslant) = 
  (\varphi A', \leqslant)$.  
  The reducts of \POA\ $(D',S',F')$-homomorphisms are defined as the
  reducts of the underlying \MSA\ $(D' \cup S', F')$-homomorphisms.   
  
\end{itemize}
\end{definition}

With all these the important `Satisfaction Lemma' of Theorem
\ref{sat-lem-msa} gets extended to \POA\ in a straightforward way. 

\begin{theorem}
For each \POA\ signature morphism $\varphi \co (D,S,F) \to
(D',S',F')$, for each $(D,S,F)$-sentence $\rho$ and each
$(D',S',F')$-algebra $(A',\leqslant)$ 
\[
(A',\leqslant) \models_{(D',S',F')} \varphi \rho
\mbox{ \ if and only if \ }
\varphi (A',\leqslant) \models_{(D,S,F)} \rho.
\]
\end{theorem}

\subsubsection*{A comparison between \POA\ and other semantic
  approaches to non-deterministic rewriting}
In \cite{meseguer92} the `rewriting logic' (abbreviated \RWL) was
introduced as a logical foundations for non-deterministic rewriting.
Then this was apparently implemented in the form of the Maude language
\cite{maude-book}.
However right from its inception \RWL\ has suffered a number of
mathematical shortcomings that propagated into a series of problems at
the more practical levels.
These culminated to the conclusion put forward by
\cite{CodescuMRM10,CodescuMRM11} that \RWL\ cannot be the logic / model
theory of Maude, and instead the only such possibility is provided by
\POA. 
But due to the historically important role played by \RWL\ it is worth to
look into the (mathematical) relationship and differences between
\RWL\ and \POA. 
\begin{enumerate}

\item While \RWL\ is mathematically complicated as it involves
  sophisticated category-theoretic structures (e.g. sub-equalisers,
  etc.), \POA\ achieves mathematical simplicity.
  In itself this may not be that important if it did not relate to
  other aspects discussed below.

\item While \RWL\ fails the Satisfaction Lemma of institution
  theory,\footnote{This failure is common knowledge among experts of
    the area. However as a statement it appears rather scarcely in the
    literature. But it can be found explicitly in \cite{iimt} and
    implicitly in \cite{CodescuMRM10,CodescuMRM11}.} all versions of
  \POA\ are institutions.
  In other words, while \POA\ belongs to the wide family of
  specification and programming logics that are captured as
  institutions \cite{ins,iimt,sannella-tarlecki-book}, \RWL\ does not.  
  The advantages of being an institution are rather vast.
  On the practical side, the design of any language that implements
  \POA\ rigorously would benefit directly from the rich institution
  theoretic modularisation theory.
  On the model theoretic side, \POA\ benefits directly from the
  in-depth axiomatic developments provided by institution
  theory.\footnote{A relevant monograph containing application to
    the model theory of \POA\ in its abridged form is
    \cite{iimt}. Since its publication the institution-theoretic
    approach to model theory has developed further.}
  Moreover, there is a strong interdependency relationship between
  the two sides.

\item In the seminal work on \RWL\ \cite{meseguer92} and also in
  subsequent works it is presented and discussed a concept of
  projection from the categories of models of \RWL\ to categories of
  models of \POA\ that essentially collapses parallel transitions
  (arrows) to a preorder relation.
  In this way a model of \RWL\ gets mapped to a preordered algebra.
  But as often happens in mathematics, simplifying structure does not
  necessarily lead to a theory that is a special case of the former,
  on the contrary it may lead to stronger and more useful properties.
  This projection has lead to a widespread mistaken belief that as a
  model theory \POA\ is a collapsed form of \RWL\ and consequently
  concepts   and results from the former can be automatically
  transfered to the latter. 
  There are rigorous arguments that refute such views such as in
  \cite{CodescuMRM10,CodescuMRM11} where, for instance, it is shown
  that the consequence relation in \POA\ is \emph{different}
  (i.e. larger) that what is obtained by collapsing \RWL.
  This situation strengthen claims of novelty regarding results
  obtained directly in \POA.\\
  The mathematical explanation for this important discrepancy is as
  follows.
  The proper way to establish the relationship between the two
  model theories is through the mathematics of institution  
  \emph{(co)morphisms} \cite{ins,goguen-rosu2000,iimt}, but in the case
  of \POA\ and \RWL\ such an attempt fails immediately due to
  the latter formalism not being an institution.
  Naturally, this discrepancy between \RWL\ and \POA\ propagates to
  computational aspects. 

\item From the side of the logical languages, in some sense \POA\ is
  significantly richer than \RWL\ because while \POA\ supports a
  fully fledged first-order language, \RWL\ sentences are restricted
  to a proper sub-class of Horn sentences. 
  For instance Horn sentences such as
  $\hornsen{X}{(t \trans t')}{(t_1 = t_2)}$ are not supported in \RWL\
  because of the hierarchical built of \RWL\ in which
  the logic of transitions is built on top of equational logic.
  But these sentences can be very important to have in the
  applications.
  For instance when verifying properties of algorithms one needs to
  express that a transition has a certain property, for instance that
  a certain (semi-)invariant holds.
  Such properties are logically encoded as equations conditioned by
  transitions.
  As an example we may consider the termination of
  \texttt{BUBBLE-SORT}.
  If for each string $s$ we let $i(s)$ denote the number of its
  inversions, then in the logical language we should be able to
  express
  \[
    \hornsen{s,s'}{(s \trans s')}{(i(s) > i(s') =
      \texttt{true})}.
  \]
  It should be noted that while such sentences are supported in
  CafeOBJ they are not supported in Maude due to the latter's design 
  commitment to \RWL. 
    
\end{enumerate}

In the recent paper \cite{meseguer2020} another semantic framework for
non-deterministic rewriting is introduced, namely
\emph{transition systems}. 
As mathematical structures these are significantly weaker than
preordered algebras as they are just algebras enhanced with binary
relations without preorder and monotonicity requirements.
So, in a sense we can say that \POA\ represents a middle
ground between \RWL\ and the semantics put forward in
\cite{meseguer2020}.

%%%%%%%%%%%%%%%%%%%%%%%%%%%%%%%%%%%%%%%%%%%%%
\section{Rewriting as Adjunction}

In this section we do the following:
\begin{enumerate}

\item We define on any algebra $B$ the rewriting relation
  $\rew_{\Gamma/B}$ determined by a set $\Gamma$ of transitions.

\item We prove that $(B,\rew_{\Gamma/B})$ is a preordered algebra that
  satisfies $\Gamma$.

\item We prove that the preordered algebras $(B,\rew_{\Gamma/B})$
  arise as a left adjoint functor, a property that can be regarded as
  an implicit denotational definition of rewriting.
  
\end{enumerate}
In this section $\Gamma$ denotes a set of $(D,S,F)$-transitions. 

\subsection{The explicit definition of algebraic rewriting}

% For the purpose of this section we fix a \POA\ signature $(D,S,F)$,
% a set $E$ of $(D \cup S, F)$-equations, a set $\Gamma$ of
% $(D,S,F)$-transitions, and a $(D \cup S,F)$-algebra $B$.

The following technical concept lies at the heart of rewriting. 

\begin{definition}[Rewriting contexts] 
Given an \MSA\  signature $(S,F)$, an $(S,F \cup \{ z \})$-term $c$
over the signature extended with a new variable $z$ 
is a \emph{(rewriting) $(S,F)$-context}
\index{rewriting context} \index{context, for rewriting}
if 
\begin{itemize}
 
\item[--] $c=z$, or
\item[--] $c= \sigma(c_1,\dots,c_n)$ such that $\sigma\in F_{w \aar
    s}$ is an operation symbol and there exists exactly one $k\in
  \{1,\dots,n\}$ such that $c_k$ is context, with $c_i$ being just
  $(S,F)$-terms for $i\not= k$.  
\end{itemize}
Then $c_k$ is called the \emph{immediate sub-context}
\index{immediate sub-context} of $c$.
% A term $c'$ is a \emph{sub-context}\index{sub-context} of (a context) $c$ 
% if it is either the immediate sub-context of $c$ or else it is a sub-context of 
% the immediate sub-context of $c$.
\end{definition}

\begin{notation}
We denote contexts by $c[z]$ where $c$ is the actual context as a
term and $z$ is the variable of the context. 
\end{notation}

\begin{definition}
Let $c[z]$ be a context of sort $s'$ such that the variable $z$ is of
sort $s$, and let $B$ be any \MSA\ algebra for the respective
signature.
The $B_c \co B_s \to B_{s'}$ is the function defined as follows:
\begin{itemize}

\item[--] If $c[z]= z$ then $B_c$ is an identity function.

\item[--] If $c = \sigma(c_1, \dots, c_k [z], \dots, c_n)$ where
  $\sigma$ is an operation symbol and $c_k [z]$ is the immediate
  sub-context of $c[z]$ then
  \[
  B_c (b) = B_\sigma (B_{c_1}, \dots, B_{c_k} (b), \dots, B_{c_n}). 
  \]
  
\end{itemize}
\end{definition}

\begin{notation}
  For any set $B$, we let $\Delta_B$ denote its \emph{diagonal}, i.e.
  $\Delta_B = \{ (b,b) \mid b \in B \}$.
  This notation extends to algebras by considering their underlying
  sets.
\end{notation}

\begin{definition}
Let $B$ be a $(D,S,F)$-algebra and let $\Omega \subseteq B \times B$. 
By $\Gamma(\Omega)$ we denote the set
\[
  \begin{array}{rl}
\{  (B_c (\theta t), B_c (\theta t')) \mid &
 c[z] \text{ applicable context\footnotemark
                                             of a \emph{system sort}, }
 \hornsen{X}{H}{(t \trans t')} \in \Gamma, \\
              & \theta \co X \to B, \
                \theta \heq \subseteq \Delta_B, 
                \theta \htr \subseteq \Omega \}.
\end{array}
\]
\footnotetext{The sort $z$ is the same with the sort of $t$ and $t'$.}
\end{definition}

In what follows the role of $\Omega$ will be played only by relations
representing bounded applications of transitions on $B$. 

The following provides the definition of a rewriting relation on $B$
determined by a set $\Gamma$ of transitions.

\begin{definition}
Let $\{ \Gamma_{B,k} \mid k \in \omega \}$ be the following
inductively defined sequence of binary relations on $B$:
\begin{itemize}

\item $\Gamma_{B,0} = \Delta_B$, 

\item $\Gamma_{B, k+1} = \Gamma_{B, k}^* \ \cup  \ \Gamma(\Gamma_{B,
    k}^*)$. 
  
\end{itemize}
For any $b,b' \in B$ we let $b \lra_{\Gamma/B} b'$ when there
  exists $k$ such that $(b,b') \in \Gamma(\Gamma_{B,k}^*)$.
  Then $\rew_{\Gamma/B}$ denotes the reflexive-transitive closure of
  $\lra_{\Gamma/B}$. 
\end{definition}

\begin{proposition}
  $b \rew_{\Gamma/B} b'$ if and only if
  $(b,b') \in \bigcup_{k\in \omega} \Gamma_{B,k}$.
\end{proposition}

\begin{proof}
  For the implication from the left to the right we consider the chain 
  \[
  b=b_1 \lra_{\Gamma/B} b_2 \lra_{\Gamma/B} \dots \lra_{\Gamma/B} b_n = b'.
\]
For each $i\in [n-1]$ there exists $k_i$ such that $(b_i, b_{i+1})
\in \Gamma_{B,{k_i}}$.
Let $k = \max\limits_{i\in [n-1]} k_i$.
Since $(\Gamma_{B,j})_{j\in \omega}$ is increasing it follows that for
each $i\in [n-1]$, $(b_i, b_{i+1}) \in \Gamma_{B,k}$, hence $(b,b') \in
\Gamma^*_{B,k} \subseteq \Gamma_{B,k+1}$.

For the implication from the right to the left we prove by induction
on $k\in \omega$ that $(b,b') \in \Gamma_{B,k}$ implies $b \rew_{\Gamma/B}
b'$.
For $k=0$ the conclusion follows by the reflexivity of $\rew_\Gamma$.
Now let $k > 0$.
\begin{itemize}

\item When $(b,b') \in \Gamma^*_{B,k}$ the conclusion follows from the
  induction hypothesis by the transitivity of $\rew_{\Gamma/B}$.

\item When $(b,b') \in \Gamma(\Gamma^*_{B,k})$ it means that $b
  \lra_{\Gamma/B} b'$ hence $b \rew_{\Gamma/B} b'$.
  
\end{itemize}
\end{proof}

The traditional implementations of term rewriting require restrictions
on the occurrences of the variables in the rewrite rules: that all
variables in the right hand side term  and in the conditions occur also
in the left hand side term.
These also have a computational significance, beyond just
implementation issues, which relates to termination.
But for the purpose of developing the results in this paper such
conditions are not required.
Moreover several implementations of rewriting go beyond such
restrictions; these include languages  / systems such as ASF+SDF,
ELAN, Maude, which support \emph{matching conditions}. 

\subsection{Model theoretic properties of algebraic rewriting} 

\begin{proposition}\label{trans-rel-poa}
$(B,\rew_{\Gamma/B})$ is a preordered $(D,S,F)$-algebra. 
\end{proposition}

\begin{proof}
Consider an operation $\sigma$ of the respective signature $(D,S,F)$
such that $\sigma \in F_{s_1 \dots s_n \aar s}$ where $s\in S$.
For each $k\in [n]$ let $b_k, b'_k \in B_{s_k}$ such that
\begin{itemize}

\item[--] $b_k = b'_k$ if $s_k \in D$, and

\item[--] $b_k \rew_{\Gamma/B} b'_k$ if $s_k \in S$. 
  
\end{itemize}
We have to prove that
\[
B_\sigma (b_1, \dots, b_n) \rew_{\Gamma/B} B_\sigma (b'_1, \dots, b'_n).
\]

In order to simplify the presentation of the argument we consider the
case $n=2$ and we assume that $s_1, s_2 \in S$.
For $n > 2$ and when some of the $s_k$ are data sorts the argument is
similar to the reduced case. 

We prove by induction on $k\in \omega$ that
\begin{quote}
  $(b_1, b'_1) \in \Gamma_{B,k}$ implies $(B_\sigma (b_1, b_2),
  B_\sigma (b'_1, b_2)) \in \Gamma_{B,k}$.
\end{quote}
For $k=0$ this is obvious.
For the induction step we assume that this holds for $k$ and prove it
for $k+1$.
So let us consider $(b_1, b'_1) \in \Gamma_{B,k+1}$.
We distinguish two cases:
\begin{itemize}

\item When $(b_1, b'_1) \in \Gamma^*_{B,k}$.
  Then there exists $b_1 = b^1_1, \dots, b^1_n = b'_1$ such that
  $(b^i_1, b^{i+1}_1) \in \Gamma_{B,k}$.
  By the induction hypothesis it follows that $(B_\sigma (b^i_1, b_2), 
  B_\sigma (b^{i+1}_1, b_2)) \in \Gamma_{B,k}$ hence
  $(B_\sigma (b_1, b_2), B_\sigma (b'_1, b_2)) \in \Gamma^*_{B,k} \subseteq
  \Gamma_{B,k+1}$. 

\item When there exists $\hornsen{X}{H}{(t \trans t')} \in \Gamma$ and
  $\theta\co X \to B$ and context $c[z]$ such
  that $\theta \heq \subseteq \Gamma_{B,0}$ and
  $\theta \htr \subseteq \Gamma^*_{B,k}$,
  $b_1 = B_c (\theta t)$ and $b'_1 = B_c (\theta t')$.
  By considering the context $c'[z] = \sigma(c[z],t_2)$  we obtain
  that $(B_\sigma (b_1, b_2), B_\sigma (b'_1, b_2)) \in
  \Gamma_{B,k+1}$.  

\end{itemize}

Since $b_1 \rew_\Gamma b'_1$ there exists $k\in \omega$ such that
$(b_1, b'_1) \in \Gamma_{B,k}$.
By the above conclusion it follows that
$(B_\sigma (b_1, b_2), B_\sigma (b'_1, b_2)) \in \Gamma_{B,k}$, hence
$B_\sigma(b_1, b_2) \rew_{\Gamma/B} B_\sigma (b'_1, b_2)$.
Similarly we can show that $B_\sigma (b'_1, b_2) \rew_{\Gamma/B}
B_\sigma (b'_1, b'_2)$.
By the transitivity of $\rew_{\Gamma/B}$ it follows that
$B_\sigma (b_1, b_2) \rew_{\Gamma/B} B_\sigma (b'_1, b'_2)$.
\end{proof}

\begin{proposition}\label{poa-rew-sat}
$(B,\rew_{\Gamma/B}) \models \Gamma$. 
\end{proposition}

\begin{proof}
  Let $\hornsen{X}{H}{(t \trans t')}$ be any sentence in $\Gamma$.
  Let $B'$ be any $(D,S,F \cup X)$-expansion of $B$ such that
  $(B',\rew_{\Gamma/B}) \models H$. 
  Let $\theta \co X \to B$ such that $\theta(x) = B'_x$ for each $x\in
  X$. Then 
  \begin{itemize}

  \item $\theta \heq \subseteq \Gamma_{B,0}$, and

  \item $\theta\htr  \ \subseteq \ \rew_{\Gamma/B}$ which
    means that there exists $k\in \omega$ such that
    $\theta\htr \subseteq \Gamma_{B,k}$.
    
  \end{itemize}
  By considering the identity context $c[z]=z$ we obtain that
  $(\theta t,\theta t') \in \Gamma_{B,k+1}$.
  Hence $\theta t \rew_{\Gamma/B} \theta t'$ which means $B'_t
  \rew_{\Gamma/B} B'_{t'}$ which means $(B', \rew_{\Gamma/B}) \models
  \lc{(t \trans t')}$. 
\end{proof}

\subsection{The adjunction property of algebraic rewriting}

\begin{notation}
For any \MSA\ signature $(S,F)$ let $\Mod(S,F)$ denote the category of
the \MSA\ $(S,F)$-algebras and their homomorphisms.
For any \POA\ signature $(D,S,F)$ we let $\Mod(D,S,F)$ denote the
category of the \POA\ $(D,S,F)$-models and their homomorphisms, and 
$\Mod(D,S,F,\Gamma)$ denote the full category of $\Mod(D,S,F)$ defined
by those preordered algebras that satisfy $\Gamma$.  
\end{notation}

The following theorem shows that any set $\Gamma$ of transitions
endows \emph{any} \MSA\ algebra with a \emph{minimal} preorder
structure that yields a \POA\ algebra that \emph{satisfies} $\Gamma$.
Moreover the original \MSA\ algebra is preserved.
Technically it is convenient to present this result as a categorical
adjunction.
The uniqueness property from this result indicates that this can be
taken as an \emph{implicit} definition of non-deterministic algebraic
rewriting. 

\begin{theorem}\label{rew-adjoint}
The forgetful functor $\mathcal{U} \co \Mod(D,S,F,\Gamma) \to
\Mod(D\cup S, F)$ has a left-adjoint left inverse which maps any $(D
\cup S,F)$-algebra $B$ to the preordered algebra
$(B,\rew_{\Gamma/B})$.  
\end{theorem}

\begin{proof}
It is enough if we proved that:
\begin{quote}
For any  preordered $(D,S,F)$-algebra $(A,\leqslant)$ such that
$(A,\leqslant)\models \Gamma$ and for any $h \co B \to A$  a $(D \cup
S,F)$-homomorphism, $h$ is a homomorphism of \POA\ algebras
$(B,\rew_{\Gamma/B}) \to (A,\leqslant)$ too.  
\end{quote}

\[
  \xymatrix @C-2em{
  B \rrto^= \drto_h & & \mathcal{U}(B,\rew_{\Gamma/B}) \dlto^h & &
  (B,\rew_{\Gamma/B}) \dlto^h \\
   & \mathcal{U}(A,\leqslant) & & (A,\leqslant) & 
  }
\]

  Thus we only have to prove that $b \rew_{\Gamma/B} b'$ implies
  $h(b) \leqslant h(b')$.
  By induction on $k \in \omega$ we prove that $(b,b') \in
  \Gamma_{B,k}$ implies $h(b) \leqslant h(b')$.

  For $k=0$ the conclusion follows by the reflexivity of $\leqslant$.

  For the induction step we assume $(b,b') \in (\Gamma_{B,k+1})^*$.
  We have two cases:
  \begin{enumerate}

  \item $(b,b') \in (\Gamma_{B,k})^*$, or

  \item $(b,b') \in \Gamma((\Gamma_{B,k})^*)$, i.e. $b = B_c
    (\theta t)$ and $b' = B_c (\theta t')$ where $c[z]$ 
    is a context, $\hornsen{X}{H}{(t \trans t')}$ belongs to $\Gamma$,
    $\theta \co X \to B$, $\theta \heq \subseteq
    \Gamma_{B,0}$ and $\theta\htr \subseteq \Gamma_{B,k}^*$. 

  \end{enumerate}
  In the former case, by the induction hypothesis $h(\Gamma_{B,k}) \
  \subseteq \ \leqslant$.
  It follows that $h(\Gamma_{B,k}^*) =
  (h(\Gamma_{B,k}))^* \ \subseteq \ \leqslant^* \ = \ \leqslant$ (the
  last equality follows by the transitivity of $\leqslant$).

  In the latter case we have $h(b) = h(B_c (\theta t)) = A_c
  (h(\theta t))$ and likewise $h(b') = A_c (h(\theta t'))$.
  Thus, by the monotonicity of the interpretations of the operations
  in $A$ it is enough to prove that $h(\theta t) \leqslant
  h(\theta t')$.

  Let us consider $A'$ the $(D,S,F \cup X)$-expansion of $A$ such that
  $A'_x = A_{h(\theta x)}$ for each $x\in X$.
  We establish that $(A', \leqslant) \models H$.
  Let $B'$ be the  $(D,S,F \cup X)$-expansion of $B$ such that $B'_x =
  \theta x$ for each $x\in X$.
  Note that $h$ is also a $(D,S,F\cup X)$-homomorphism $B' \to
  A'$.
  \begin{itemize}

  \item Let $(t_1 = t_2) \in \heq$.
    Then $(A'_{t_1}, A'_{t_2}) = (h(B'_{t_1}), h(B'_{t_2})) =
    (h(\theta t_1), h(\theta t_2)) \in h(\theta \heq)
    \subseteq h(\Gamma_{B,0})$.
    Hence $A'_{t_1} = A'_{t_2}$ which means $A' \models (t_1 = t_2)$.

  \item Let $\lc{(t_1 \trans t_2)} \in \htr$.
    Then, similarly to above we get that $(A'_{t_1}, A'_{t_2}) \in
    h(\theta\htr) \subseteq h(\Gamma_{B,k}^*)$.
    From the proof at the former case of the induction step we know
    that $h(\Gamma_{B,k}^*) \ \subseteq \ \leqslant$.
    Hence $A'_{t_1} \leqslant A'_{t_2}$ which shows that
    $A' \models (t_1 \lc{\trans} t_2)$.  

  \end{itemize}
  Finally, since $(A,\leqslant) \models \hornsen{X}{H}{(t \trans t')}$
  and   $(A',\leqslant) \models H$ it follows that $(A',\leqslant)
  \models (t \lc{\trans} t')$ which means $A'_t \leqslant A'_{t'}$
  which means $h(\theta t) \leqslant h(\theta t')$.  
\end{proof}

The unit given by the adjunction given by this theorem is the
identity.
In the literature, eg. \cite{ins,iimt} such adjunctions are called
\emph{strongly persistent} adjunctions.
They are stronger than the ordinary adjunctions but weaker than the 
equivalences of categories.

\begin{example}\label{bubblesort-ex2}
  We continue Example \ref{bubblesort-ex}.
  Recall that from the three \POA\ $(D,S,F)$-models presented there
  only the first two satisfy the transition of the specification.
  From these two the second one is the free $(D,S,F,\Gamma)$-model
  over the underlying $(D \cup S,F)$-algebra of the lists of natural
  numbers. 
\end{example}

\begin{example}\label{sol-ex2}
  We continue Example \ref{sol-ex}.
  There $D = \emptyset$.
  From the three $\POA$ models presented in Example \ref{sol-ex}, all
  of them sharing the same underlying \MSA\ $(S,F)$-algebra, only $A$
  and $B$ satisfy the transition of \texttt{SOLUTIONS}.
  From these two $A$ is the free model over the underlying \MSA\
  algebra.

  This example displays a kind of non-determinism that is not
  available in the term based approaches to rewriting, neither in
  \RWL\ nor on \POA. 
  For instance we have that $4 \leqslant_A 2$ and $4 \leqslant_A -2$
  because the equation $x^2 = 4$ has $2$ and $-2$ as solutions.
  In the approaches that deal only with terms, a single transition
  can be applied only in one way to a term because term rewriting
  matching is deterministic.
  In our example this matching corresponds to solving quadratic
  equations $x^2 = a$, which may have two solutions.
  Implementation of algebraic rewriting over predefined types may
  replace the matching algorithms by constraint solvers.  
\end{example}  

\begin{notation}
It is well known that any equational theory $E$ admits an initial
algebra.
Let us denote it by $0_E$.
Given a set $\Gamma$ of transitions we abbreviate by $\rew_{\Gamma/E}$
the rewriting relation $\rew_{\Gamma, 0_E}$ on $0_E$ determined by
$\Gamma$.
\end{notation}
We have the following practically important consequence of Theorem
\ref{rew-adjoint}.  

\begin{corollary}\label{poa-init}
$(0_E, \rew_{\Gamma/E})$ is the initial preordered algebra that
satisfies $E \cup \Gamma$. 
\end{corollary}

\begin{proof}
$0_E \models E$ and moreover $(0_E, \rew_{\Gamma/E}) \models \Gamma$
by Proposition \ref{poa-rew-sat}.
Let $(A,\leqslant)$ be any preordered $(D,S,F)$-algebra such that
$(A,\leqslant) \models E \cup \Gamma$.
Then $A \models E$.
By the initiality property of $0_E$ there exists an unique $(D \cup
S,F)$-homomorphism $h \co 0_E \to A$.
By Theorem \ref{rew-adjoint} we have that $h$ is also a
$(D,S,F)$-homomorphism $(0_E, \rew_{\Gamma/E}) \to (A,\leqslant)$. 
\end{proof}

Modulo \RWL\ restrictions of the occurences of variables,
the result of Corollary \ref{poa-init} corresponds to a related
initiality result in \cite{meseguer92}.
However due to the collapse of \RWL\ to \POA\ being illusory, as
discussed in Section \ref{poa-institution}, the derivation of
Corollary \ref{poa-init} from the initiality result of
\cite{meseguer92} appears as problematic.
Moreover, the source of this corollary, namely Theorem
\ref{rew-adjoint}, does not have any correspondent in the \RWL\
literature. 
An important class of examples of this refer to rewriting on
predefined types, something which is unavoidable in any proper
programming  language, Maude and CafeOBJ included.
For instance both Examples \ref{bubblesort-ex2} and \ref{sol-ex2} fall
outside the scope of the initial semantics result of
Corollary \ref{poa-init}.

\begin{example}\label{bubble-sort-user-defined}
The \RWL\ literature abounds of examples that fall within
the scope of the initial semantics result of \cite{meseguer92}, and
many of those can be made examples for our Corollary \ref{poa-init}.
However let us provide here a concrete example based on
\texttt{BUBBLE-SORT}.
For this we have to replace the predefined \texttt{NAT} and
\texttt{BOOL} by user defined \texttt{NAT} and \texttt{BOOL} with
initial equational semantics. 
We can get minimalist about that and consider the following
specification.
\begin{maude}
mod NAT is 
  sorts Bool Nat .
  ops true false : -> Bool .
  op 0 : -> Nat .
  op s_ : Nat -> Nat .
  vars n m : Nat .
  eq (0 < s n) = true .
  eq (n < 0) = false .
  eq (s m < s n) = (m < n) .
endm
\end{maude}
Then $E$ of Corollary \ref{poa-init} consists of the three equations
of \texttt{NAT} above plus the associativity equation of lists
concatenation.
$0_E$ is precisely the \MSA\ algebra of the lists of naturals.
Consequently $\rew_{\Gamma/E}$ corresponds to the preorder of the
second model presented in Example \ref{bubblesort-ex}. 
\end{example}

%%%%%%%%%%%%%%%%%%%%%%%%%%%%%%%%%%%%%%%%%%%%%%
\section{Computing in preordered algebras}
\label{comput-sec}

Deterministic term rewriting comes with some clear and practically
relevant sufficient conditions (i.e. confluence and termination) for
the completeness of computations.
Under these conditions the completeness is explained in various
different ways that suit different backgrounds.
The highest and the mathematically most advanced such explanation is
the isomorphism between, on the one hand the initial algebra of the
corresponding equational theory, and on the other hand the algebra of
the normal forms of the rewriting relation on terms.
This applies equally to plain term rewriting and to term rewriting
modulo a theory.
The former model can be considered as the \emph{denotational model}
and the latter as the \emph{computational model}.\footnote{Note that
  here the term ``model'' is used in the precise mathematical sense
  given by model theory rather than in the informal sense given by the
  concept of ``modelling''.}

This approach brings a novelty to the theory of non-deterministic
rewriting.
Most of the \RWL\ literature around the language Maude
(e.g. \cite{maude-book}) defines a computational method for
non-deterministic rewriting on initial algebras modulo equational
theories, defines its completeness in proof theoretic terms, and
provides sufficient conditions for this completeness. 
Moreover all these are presented in a colloquial style
\cite{maude-book} lacking mathematical rigor. 
However in \cite{meseguer2020} the so-called ``coherence problem'' is
defined in terms that are similar to our model homomorphisms approach,
but modulo the crucial and rather ample difference between the model
theoretic structures involved.   
Our aim in this section is to fill this gap for this computational
method, first by providing a computational model in a model theoretic
sense within the framework of \POA, and then by specifying a set of
sufficient conditions that guarantee its existence and its soundness
and completeness.
Because the computational method refers to non-deterministic rewriting
modulo equational theories our completeness result is built on top of
the above mentioned model theoretic completeness of deterministic
rewriting. 

More specifically, in this section we do the following.
\begin{enumerate}

\item We present the method to compute with \POA\ transitions modulo
  equations that is employed by languages such as Maude and CafeOBJ.
  This constitutes the main application for the theory developed in
  this section. 

\item We define an abstract \POA\ model that serves as a generic
  ``computational'' model. 

\item Under some conditions that match very well the practice of
  programming in Maude and CafeOBJ we prove the isomorphism between
  the computational model and the denotational model given by the
  adjointness result of Theorem \ref{rew-adjoint}.  
  This isomorphism result explains the soundness and the completeness
  of the respective computational method.   
  
\end{enumerate}

\subsection{The computational method}
In this section we present the method employed by Maude and CafeOBJ
for hybrid computing with equations and transitions.
However our presentation is tailored to the framework of this paper.

For a program consisting of a set $E$ of equations and a set $\Gamma$
of transitions, in principle there are two different levels of
computing modulo axioms.  
\begin{enumerate}

\item An ``upper'' level where the rewriting relation is determined by 
  $\Gamma$ on the initial algebra $0_E$.
  This is rewriting by $\Gamma$ modulo $E$. 

\item A ``lower'' level where $E$ is partitioned as $E = E_0 \cup
  E_1$, with $E_1$ being used as rewrite rules modulo $E_0$.
  Thus $0_E$ is obtained as the algebra of the normal forms of
  rewriting modulo $E_0$ by $E_1$, or otherwise said the normal forms
  of rewriting by $E_1$ in $0_{E_0}$.
  At this level, with respect to $E_0$ Maude and CafeOBJ implement
  only associativity, commutativity, identity of user defined
  \emph{binary operations}.  
  
\end{enumerate}

Thus with a program consisting of a set $E$ of equations and a
set $\Gamma$ of transitions, computations involve two rewriting
systems, one for $E_1$ (modulo $E_0$), and one for $\Gamma$ (also
modulo $E_0$). 
Here there is the basic assumption that
\begin{quote}
  \hspace{-2em}
  \emph{the rewrite system modulo $E_0$ determined by $E_1$ is
    confluent and terminating,} 
\end{quote}
so each term $t$ modulo $E_0$,\footnote{The congruence class of $t$
  determined by $E_0$ on the term algebra.} i.e., $t_{E_0} \in 0_{E_0}$,
has an unique normal form $\sem{t_{E_0}}$ with respect to
$\rew_{E_1 / E_0}$.

\begin{example}\label{bubble-e01}
Let us consider \texttt{BUBBLE-SORT} with user defined data types
(like in Example \ref{bubble-sort-user-defined}).
Then $E_0, E_1, \Gamma$ are as follows:
\begin{itemize}

\item $E_0$ consists of the associativity of list concatenation.

\item $E_1$ consists of the three equations of \texttt{NAT} as
  specified in Example \ref{bubble-sort-user-defined}. 

\item $\Gamma$ consists of the single transition. 
  
\end{itemize}
Then $E_1$ is confluent and terminating on the sorts \texttt{Nat} and
\texttt{Bool}, the normal forms of sort \texttt{Nat} being the natural
numbers in the form \texttt{(s s...s 0)}, and of sort \texttt{Bool}
being just the two constants.
This extends immediately to the confluence and termination modulo
$E_0$, that takes place on the sort \texttt{List\{Nat\}}, because in
the terms the list concatenation operation \texttt{\_\_} never occurs
below an operation from \texttt{NAT}. 
\end{example}

Then a computation process that starts with a term $t$ is governed by
the following \emph{non-deterministic} algorithm:

\

\begin{tcolorbox}
\begin{itemize}
 
\item[0.] The current term $T$ is set to $t$.

\item[1.] $T$ is reduced to its normal form $\sem{T}$ under
  $\rew_{E_1 / E_0}$.  

\item[2.] The new value of $T$ is set to a term obtained from
  $\sem{T}$ by applying one rewrite step by $\Gamma$.
  Then  we move to step 1. 
 
\end{itemize}
\end{tcolorbox}

\

\noindent
This can be considered as a computational modelling of the initial
preordered algebra $(0_E, \rew_{\Gamma/E})$ (Corollary \ref{poa-init}).
In what follows we provide an argument for this. 

According to the theory of equational rewriting, in the presence of
the assumption of termination and confluence, the initial algebra 
$0_E$ and the algebra $N_{E_1 / E_0}$ of the normal forms of the
rewriting relation $\rew_{E_1 / E_0}$ on $0_{E_0}$ are isomorphic. 
The mathematical modelling of the algorithm above that defines
computation in \POA\ is given by the reflexive-transitive
$\rew_{\Gamma/\sem{\_}}$ closure of the following relation on
$N_{E_1 / E_0}$: 
\[
  \sem{t_{E_0}} \lra_{\Gamma/\sem{\_}} \sem{t'_{E_0}}
  \mbox{ \ iff there exists \ }
  u \mbox{ \ such that \ }
  \sem{t_{E_0}} \lra_{\Gamma/E_0} u_{E_0}
  \mbox{ \ and \ }
  \sem{u_{E_0}} = \sem{t'_{E_0}}. 
\]

Thus our goal is
\begin{itemize}

\item first, to establish that $(N_{E_1 / E_0},
  \rew_{\Gamma/\sem{\_}})$ is a preordered algebra, and

\item second, to establish that the isomorphism $h \co 0_E \to
N_{E_1 / E_0}$ gives an isomorphism of preordered algebras
\begin{equation}\label{poa-comp}
(0_E, \rew_{\Gamma/E}) \ \cong \ (N_{E_1 / E_0},\rew_{\Gamma/\sem{\_}}).
\end{equation}

\end{itemize}

% In computing science terminology \eqref{poa-comp} represents the
% equivalence between the \emph{denotational} and the \emph{operational 
%   semantics} of the program given by $E_0, E_1$ and $\Gamma$. 
In order to establish \eqref{poa-comp} it is enough to prove two
things:
\begin{equation}\label{thing1}
  \mbox{that \ } h(\rew_{\Gamma/E}) \ \subseteq  \
  \rew_{\Gamma/\sem{\_}}
  \mbox{, and}
\end{equation}
\vspace{-1.5em}
\begin{equation}\label{thing2}
  \mbox{that \ } h^{-1} (\rew_{\Gamma/\sem{\_}}) \ \subseteq \ \rew_{\Gamma/E}. 
\end{equation}
The latter fact, which is the easier one, represents a sort of
\emph{soundness} property, because it says that through our algorithm,
from each state $t_E$ we can reach only states $t'_E$ such that
$t_E \rew_{\Gamma/E} t'_E$.
The former fact, more difficult to establish, represents a sort of
\emph{completeness} aspect: if $t_E \rew_{\Gamma/E} t'_E$ then $t'_E$
can be computed from $t_E$ by our algorithm.

\subsection{The abstract computational \POA\ model}

In order to study the properties \eqref{thing1} and \eqref{thing2}
it is mathematically convenient to abstract $0_{E_0}$, $0_E$ and
$N_{E_1 / E_0}$ to arbitrary algebras; in this way we achieve simpler
notations, we get rid of some concrete redundant details (such as
terms and congruence classes of terms) and by those we are able to
unhide the essential causalities.
Moreover this abstract framework can be also applied to other
particular situations of interest such as rewriting-based computations
with predefined types.
In fact all these are well known general benefits of axiomatic
treatment of problems.
In this case the mathematical object thus obtained will be an abstract
\POA\ model that represents the computational method. 

In what follows we will often switch between \MSA\ and \POA\ algebras
and back.
This is based on an implicit assumption of a \POA\ signature $(D,S,F)$
and of its associated \MSA\ signature $(D \cup S,F)$.

The following concept captures abstractly the essence of the normal
forms with respect to rewriting relations which has two aspects: as a
mapping, the computation of normal forms is homomorphic and
idempotent.   

\begin{definition}\label{nf-homomorphism} An \MSA\ homomorphism
$\beta \co B \to N$ is an \emph{nf-homomorphism} when
$N \subseteq B$ and $N$ is invariant with respect to $\beta$, i.e.,
$\beta n = n$ for each $n \in N$. 
\end{definition}

\begin{example}[Normal forms of common term rewriting]
We let $B=0_{(S,F)}$ be the term algebra for an \MSA\ signature
$(S,F)$ and $E$ be a set of $(S,F)$-equations which we use for
rewriting. 
If the rewriting relation $\rew_{E/0_{(S,F)}}$ is confluent and
terminating then the normal forms $\sem{t}$ of the terms $t$ form an
algebra $N_E$ (this is $N$ of the Definition \ref{nf-homomorphism})
that is actually isomorphic to the initial algebra $0_E$.
The mapping of the terms to their normal forms, i.e. $t \mapsto
\sem{t}$ is a homomorphism $0_{(S,F)} \to N_E$, so $\beta$ is
$\sem{\_}$. 
\end{example}

\begin{example}[Normal forms of term rewriting modulo axioms]
  We let $B= 0_{E_0}$ be the initial algebra of an equational theory
  $E_0$ for an \MSA\ signature $(S,F)$.
  Let $E_1$ be a set of $(S,F)$-equations used for rewriting modulo
  $E_0$.
  If the rewriting relation $\rew_{E_1 / E_0}$ (on $0_{E_0}$) is
  confluent and terminating then the normal forms $\sem{t_{E_0}}$ of
  the terms $t_{E_0}$ modulo $E_0$ form the algebra that we denoted
  $N_{E_1 / E_0}$ (this is $N$ of Definition \ref{nf-homomorphism}).
  Then the mapping $\sem{\_} \co 0_{E_0} \to N_{E_1 / E_0}$ is an
  \MSA\ homomorphism.
\end{example}

The following gives an abstract mathematical definition of the above
non-deterministic algorithm that combines equational with
transition-based rewriting.
What it essentially says is that the rewriting relation projected on
the normal forms should be obtained through a finite sequence of
steps of the algorithm defining the computation method.
The notation $\beta(\lra_{\Gamma/B})$ stands for this projection,
i.e. the relation
$\{ (\beta b_1, \beta b_2) \mid b_1 \lra_{\Gamma/B} b_2 \}$. 

\begin{definition}
Let $\Gamma$ be a set of transitions, $\beta \co B \to N$ be an
nf-homomorphism and $\lra_{\Gamma/\beta}$ be the following binary
relation on $N$:  
\[
n_1 \lra_{\Gamma/\beta} n_2 \mbox{ \ if and only if there exists \ }
b_2
\mbox{ \ such that \ } n_1 \lra_{\Gamma/B} b_2 \mbox{ \ and \ }
\beta b_2 = n_2. 
\]
We say that \emph{$\Gamma$ is $\beta$-coherent} when
\[
\beta(\lra_{\Gamma/B}) \ \subseteq \ \rew_{\Gamma/\beta}.
\]
\end{definition}

\begin{example}\label{coh-example}
When $B = 0_{E_0}$, $N = N_{E_1 /  E_0}$, $\beta = \sem{\_}$, that
$\Gamma$ is $\sem{\_}$-coherent means that whenever $t_{E_0}
\lra_{\Gamma/E_0} t'_{E_0}$ there exists $u$ such that $\sem{t_{E_0}}
\rew_{\Gamma/E_0} u_{E_0}$ and $\sem{u_{E_0}} = \sem{t'_{E_0}}$. 
\end{example}

With the next two examples we get more concrete about the
$\beta$-coherence condition.

\begin{example}\label{bubble-beta-coh}
Within the context of the Examples \ref{bubble-e01} and
\ref{coh-example} combined we have that $B = 0_{E_0}$ is the algebra
of the lists of $\Sigma_{\texttt{NAT}}$-terms, where
$\Sigma_{\texttt{NAT}}$ is the signature \texttt{NAT} of Example
\ref{bubble-sort-user-defined}.
In $B$, on the sort \texttt{Bool}. besides \texttt{true} and
\texttt{false} we have also the potential inequalities between natural
numbers.
$N = N_{E_1 /  E_0}$ is the algebra of the lists of the natural
numbers, that share with $B$ the same interpretations of the sorts
\texttt{Nat} and \texttt{List\{Nat\}}, but on \texttt{Bool} we have
only the two constants, \texttt{true} and \texttt{false}.
The homomorphism $\beta$ essentially reduces any potential inequality
between natural numbers to any of the two constants.
The $\beta$-coherence conditions holds trivially true because
$\lra_{\Gamma/B}$ (aka $\lra_{\Gamma,E_0}$) is empty as the transition
of $\Gamma$ cannot be applied because the equations responsible for
the evaluations of the Boolean condition belong to $E_1$ rather than
$E_0$.

However we can twist this example such that the $\beta$-coherence
condition is not trivial anymore.
First we move the three equations of \texttt{NAT} from $E_1$ to $E_0$.
If we did just that then the $\beta$-coherence condition would still
be trivial but in a different way than  before; now $\beta$ would be
an identity homomorphism because $E_1$ gets emptied.
However we can still escape this triviality by adding more operations
on the natural numbers, such as an addition operation which is
specified as follows:
\begin{maude}
  op _+_ : Nat Nat -> Nat .
  eq m + 0 = m .
  eq m + s n = s(m + n) .
\end{maude}
Then $E_1$ would consist of the two equations above.
Now $\beta$ essentially evaluates additions.
Let us consider the $\beta$-coherence condition as given in Example
\ref{coh-example}.
Then $t_{E_0} \lra_{\Gamma/ E_0} t'_{E_0}$ means that two adjacent
natural numbers from $t_{E_0}$, \emph{in their normal form}, are
swapped.
The rest of the elements of $t_{E_0}$ may not necessarily be free of
additions.
In this case $\sem{t_{E_0}}$ is the list of the natural numbers
obtained by evaluating all additions in $t_{E_0}$ and $u_{E_0}$ is the
same with respect to $t'_{E_0}$.
Obviously $\sem{u_{E_0}} = u_{E_0}$ and moreover
$\sem{t_{E_0}} \lra_{\Gamma/ E_0} u_{E_0}$. 
\end{example}

The next example illustrates a failure of the $\beta$-condition.
It is adapted after an example from \cite{maude-book}. 

\begin{example}\label{beta-fail}
Let us consider the following specification:
\begin{maude}
mod BETA-FAIL is 
  sorts U B M .
  ops a b c : -> U .
  ops 0 1 : -> B .
  op [__] : U B -> M .
  op __ : M M -> M [assoc comm] .
  eq [a 0] [b 1] [c 0] = [a 0] [b 1] .
  eq [a 1] [b 0] [c 1] = [a 1] [b 0] .
  rl [b 0] => [b 1] .
  rl [b 1] => [b 0] .
  rl [a 0] [b 0] => [a 1] [b 1] .
  rl [a 1] [b 1] => [a 0] [b 0] .
endm
\end{maude}
We set $E_0$ to the associativity and the commutativity of the union
of multisets, $E_1$ to the set of the two explicit equations, and
$\Gamma$ to the set of the four transitions. 
Let $t_{E_0}$ be \texttt{[a 0] [b 1] [c 0]} and $t'_{E_0}$ be
\texttt{[a 0] [b 0] [c 0]}.
Then $t_{E_0} \lra_{\Gamma/ E_0} t'_{E_0}$ through the application of
the second transition.
We also have $\sem{t_{E_0}} = \texttt{[a 0] [b 1]}$ and
$\sem{t'_{E_0}} = t'_{E_0}$.
But any $u_{E_0}$ such that $\sem{t_{E_0}} \rew_{\Gamma/ E_0} u_{E_0}$
does not contain \texttt{c}, so we cannot possibly have
$\sem{u_{E_0}} = \sem{t'_{E_0}} (= t'_{E_0})$. 
\end{example} 

The \POA\ model $(N,\rew_{\Gamma/\beta})$ obtained through Proposition
\ref{nf-poa} below is an abstraction of the envisaged computational
model $(N_{E_1 / E_0},\rew_{\Gamma/\sem{\_}})$.

\begin{proposition}\label{nf-poa}
Let $\beta \co B \to N$ be an nf-homomorphism and $\Gamma$ be a 
$\beta$-coherent set of transitions.
Then  $(N,\rew_{\Gamma/\beta})$ is a preordered
algebra and moreover $\beta \co (B,\rew_{\Gamma/B}) \to
(N,\rew_{\Gamma/\beta})$ is a \POA\ homomorphism.  
\end{proposition}

\begin{proof}
We have to prove that the interpretation of the operation symbols by
$N$ are monotone with respect to $\rew_{\Gamma/\beta}$.
Let us consider an operation symbol $\sigma$ of a system sort.
Without any loss of generality, in order to simplify the presentation,
we assume that $\sigma$ takes precisely two arguments.
We first prove that:
\begin{equation}\label{nf-poa-e1}
  n_1 \lra_{\Gamma/\beta} n_2 \mbox{ \ implies \ }
  N_\sigma (n_1, n) \rew_{\Gamma/\beta} N_\sigma (n_2, n).
\end{equation}
By hypothesis there exists $m_2$ such that $n_1 \lra_{\Gamma/B} m_2$
and $\beta m_2 = n_2$.
Then
\begin{proofsteps}[stepsep=-2pt, justwidth=15em]
\step[beta1]{$B_\sigma (n_1, n) \rew_{\Gamma/B} B_\sigma (m_2,
  n)$}{$B_\sigma$ monotone with respect to $\rew_{\Gamma/B}$}
\step[beta2]{$\beta(B_\sigma (n_1, n)) \rew_{\Gamma/\beta} \beta(B_\sigma (m_2,
  n))$}{from \ref{beta1} since $\beta(\rew_{\Gamma/B}) \ \subseteq \
  \rew_{\Gamma/\beta}$ because $\Gamma$ is $\beta$-coherent}
\step[beta3]{$N_\sigma (\beta n_1, \beta n) \rew_{\Gamma/\beta} N_\sigma
  (\beta m_2, \beta n)$}{from \ref{beta2} since $\beta$ is \MSA\
  homomorphism}
\step[beta4]{$N_\sigma (n_1, n) \rew_{\Gamma/\beta} N_\sigma (n_2, n)$}{from
  \ref{beta3} since $\beta m_2 = n_2$ and $N$ is invariant wrt $\beta$.}
\end{proofsteps}
By transitivity we extend \eqref{nf-poa-e1} to
\begin{equation}\label{nf-poa-e2}
  n_1 \rew_{\Gamma/\beta} n_2 \mbox{ \ implies \ }
  N_\sigma (n_1, n) \rew_{\Gamma/\beta} N_\sigma (n_2, n).
\end{equation}
In a similar way we may prove that
\begin{equation}\label{nf-poa-e3}
  n_1 \rew_{\Gamma/\beta} n_2 \mbox{ \ implies \ }
  N_\sigma (n, n_1) \rew_{\Gamma/\beta} N_\sigma (n, n_2).
\end{equation}
Now let us consider $n_1, n'_1, n_2, n'_2$ arguments for $N_\sigma$
such that $n_1 \rew_{\Gamma/\beta} n_2$ and $n'_1 \rew_{\Gamma/\beta} n'_2$.
We have that:
\begin{proofsteps}[stepsep=-2pt, justwidth=15.5em]
\step[nf1]{$N_\sigma (n_1, n'_1) \rew_{\Gamma/\beta} N_\sigma (n_2,
  n'_1)$}{from \eqref{nf-poa-e2}}
\step[nf2]{$N_\sigma (n_2, n'_1) \rew_{\Gamma/\beta} N_\sigma (n_2,
  n'_2)$}{from \eqref{nf-poa-e3}}
\step[nf3]{$N_\sigma (n_1, n'_1) \rew_{\Gamma/\beta} N_\sigma (n_2,
  n'_2)$}{from \ref{nf1} and \ref{nf2} by the transitivity of
  $\rew_{\Gamma/\beta}$.} 
\end{proofsteps}
Thus $(N,\rew_{\Gamma/\beta})$ is a preordered algebra.
In addition to that, from the $\beta$-coherence assumption it follows 
that $(B,\rew_{\Gamma/B}) \to (N,\rew_{\Gamma/\beta})$ is a \POA\
homomorphism.   
\end{proof}

\begin{corollary}
If $\Gamma$ is $\sem{\_}$-coherent then $(N_{E_1 /   E_0},
\rew_{\Gamma/\sem{\_}})$ is a preordered algebra and $\sem{\_} \co
(0_{E_0}, \rew_{\Gamma/E_0}) \to (N_{E_1 /   E_0},
\rew_{\Gamma/\sem{\_}})$ is a \POA\ homomorphism. 
\end{corollary}

The role played by the $\beta$-coherence is to establish
$(N,\rew_{\Gamma/\beta})$ as a \POA\ model.
The following example shows that in its absence this may fail to be a
\POA\ model.

\begin{example}
Let us consider the example of the failure of $\beta$-coherence given
by \texttt{BETA-FAIL} of Example \ref{beta-fail}.
Then we have that
\begin{equation}\label{a0b0c0}
\texttt{[a 0] [b 0] [c 0]} \lra_{\Gamma/\beta} \texttt{[a 0] [b 1]}.
\end{equation}
because
$\texttt{[a 0] [b 0] [c 0]} \lra_{\Gamma/B}
\texttt{[a 0] [b 1] [c 0]}$ and\\
$\sem{\texttt{[a 0] [b 1] [c 0]}} = \texttt{[a 0] [b 1]}$.
Let us add $\texttt{[a 1] [c 1]}$ to each of the two multisets of
\eqref{a0b0c0}.
If the union of the multisets were monotone with respect to
$\rew_{\Gamma/\beta}$ then we would have that
\[
  \sem{\texttt{[a 0] [a 1] [b 0] [c 0] [c 1]}} \rew_{\Gamma/\beta}
  \sem{\texttt{[a 0] [a 1] [b 1] [c 1]}}
\]
which means
\[
  \texttt{[a 0] [a 1] [b 0] [c 0]} \rew_{\Gamma/\beta}
  \texttt{[a 0] [a 1] [b 1] [c 1]}.
\]
But this is not possible because there is no way to transform
\texttt{[c 0]} into \texttt{[c 1]}.
Therefore the union of the multisets is \emph{not} monotone with
respect to $\rew_{\Gamma/\beta}$. 
\end{example}

\subsection{The soundness of the operational semantics}

\begin{proposition}\label{soundness}
Under the hypotheses of Proposition \ref{nf-poa} if $\beta' \co N
\to B'$ is an \MSA\ homomorphism then  
$\beta' \co (N,\rew_{\Gamma/\beta}) \to (B',\rew_{\Gamma/B'})$ is
a \POA\ homomorphism. 
\end{proposition}

\begin{proof}
In this case we have only to prove that 
\begin{equation}\label{beta'}
\beta'(\rew_{\Gamma/\beta}) \ \subseteq \ \rew_{\Gamma/B'}. 
\end{equation}  
Let $n_1 \lra_{\Gamma/\beta} n_2$ and let $b_2 \in B$ such that $n_1
\lra_{\Gamma/B} b_2$ and $\beta b_2 = n_2$.
Since $\beta' \circ \beta \co B \to B'$ is an \MSA\ homomorphism, by
virtue of Theorem \ref{rew-adjoint} we have that
\[
\beta' \circ \beta \co (B,\rew_{\Gamma/B}) \to (B',\rew_{\Gamma/B'})
\]
is a \POA\ homomorphism.
Hence 
\[
\beta'(\beta n_1) \rew_{\Gamma/B'} \beta'(\beta b_2)
\]
which by the invariance assumption and since $\beta b_2 = n_2$ means
\[
\beta' n_1 \rew_{\Gamma/B'} \beta' n_2. 
\]
Hence $\beta'(\lra_{\Gamma/\beta}) \ \subseteq \ \rew_{\Gamma/B'}$.
This can be extended to \eqref{beta'} by using the transitivity of 
$\rew_{\Gamma/B'}$. 
\end{proof}

The property \eqref{thing2} now follows immediately:

\begin{corollary}[The concrete soundness of the computational
  method]\label{sound-poa-comp}  
  If $\Gamma$ is a $\sem{\_}$-coherent then
  $h^{-1} (\rew_{\Gamma/\sem{\_}}) \ \subseteq \ \rew_{\Gamma/E}$.  
\end{corollary}

\begin{proof}
In Proposition \ref{soundness} we set $B = 0_{E_0}$, $B' = 0_E$, $N =
N_{E_1 / E_0}$, $\beta ' = h^{-1}$ (where $h$ is the isomorphism $0_E
\to  N_{E_1 / E_0}$ , $\beta = \sem{\_}$ (the nf-homomorphism
$0_{E_0} \to N_{E_1 / E_0}$).
\end{proof}

\subsection{The completeness of the operational semantics}

It is common in model theory that the completeness properties are more 
difficult to establish than the soundness properties.
This is also the case with the completeness of the \POA\ operational
semantics with respect to the denotational semantics, which needs some
additional conditions. 

\begin{proposition}\label{comp-poa-prop}
Let $\beta \co B \to N$ be an nf-homomorphism and let $\Gamma$ be a
$\beta$-coherent set of transitions that in addition satisfy the
following conditions:
\begin{itemize}

\item[--]
  For each transition $\lc{\forall{X} H \implies (t \trans t')}$ in
  $\Gamma$, $H$ is a finite conjunction of equations of the form $t_i
  = c_i$ where $c_i$ is a constant such that $B_{c_i} \in N$.

\item[--] For each sort $s$ of an equation $t_i = c_i$ that occurs in
  a condition of a transition in $\Gamma$ (like above):
  \[
   \begin{array}{lr}
   B \models \lc{\forall{x} \bigor \{ x = c \mid c\in F_{\to s}, \
     B_c \in N \}} &
     \mbox{(no junk)} \\[3pt]
     B \models \lc{\bigand} \{ c \neq c' \mid c \neq c' \in F_{\to
     s}, \ B_c, B_{c'} \in N  \} &
     \mbox{(no confusion)}  
   \end{array}                           
  \]
 
\end{itemize}
Then $(N,\rew_{\Gamma/\beta}) \models \Gamma$. 
\end{proposition}

\begin{proof}
Let $\lc{\forall{X} H \implies (t \trans t')} \in \Gamma$ and $\ol{N}$
be an expansion of $N$ such that $\ol{N} \models H$. 
Let $(t_i = c_i) \in \ol{H}$.
Then $\ol{N} \models (t_i = c_i)$.
Let $\ol{B}$ be the expansion of $B$ such that for each variable $x\in
X$ we have $\ol{B}_x = \ol{N}_x$.
This determines also an expansion $\ol{\beta}$ of $\beta$.
Let us prove that $\ol{B} \models (t_i = c_i)$.

According to the conditions there exists a constant $c$ such that
$B_c \in N$ and $\ol{B}_{t_i} = B_c$.
By \emph{Reductio ad Absurdum} suppose that $c \neq c_i$.
Since $\ol{B}$ is a homomorphism we have that
\[
\ol{N}_{t_i} = \ol{\beta}(\ol{B}_{t_i}) = \beta(B_c) = N_c = \ol{N}_c
\]
hence $N_c = N_{c_i}$.
It follows that $B_c = N_c = N_{c_i} = B_{c_i}$ which contradicts the
conditions of the proposition.
Thus $c=c_i$ which means $\ol{B}_{t_i} = B_{c_i}$ which means $\ol{B}
\models (t_i = c_i)$.
Thus $\ol{B} \models H$.
Since $(B,\rew_{\Gamma/B}) \models \Gamma$ it follows that
$(\ol{B},\rew_{\Gamma/B}) \models \lc{(t \trans t')}$ which means
$\ol{B}_t \rew_{\Gamma/B} \ol{B}_{t'}$.

By Proposition \ref{nf-poa} it follows that $\ol{\beta}(\ol{B}_t)
\rew_{\Gamma/\beta} \ol{\beta}(\ol{B}_{t'})$.
By the homomorphism property of $\ol{\beta}$ we have that
$\ol{\beta}(\ol{B}_t) = \ol{N}_t$ and $\ol{\beta}(\ol{B}_{t'}) =
\ol{N}_{t'}$ hence $\ol{N}_t \rew_{\Gamma/\beta} \ol{N}_{t'}$. 
\end{proof}

\begin{corollary}\label{comp-poa-cor}
  Let $q \co B \to B'$ and $h \co B' \to N$ be \MSA\ homomorphisms.
  Let us assume that $\beta = h \circ q$ is an nf-homomorphism.
  Let $\Gamma$ be a set of transitions satisfying the conditions of
  Proposition \ref{comp-poa-prop}.
  Then $h \co (B',\rew_{\Gamma/B'}) \to (N,\rew_{\Gamma/\beta})$ is a
  \POA\ homomorphism. 
\end{corollary}

\begin{proof}
  By Proposition \ref{comp-poa-prop} we have that
  $(N,\rew_{\Gamma/\beta})\models \Gamma$.
  The conclusion follows by Theorem \ref{rew-adjoint}. 
\end{proof}

\begin{corollary}[The completeness of the concrete computational
  method]\label{comp-poa-comp}  
  We assume the following conditions on $\Gamma$:
  \begin{itemize}

  \item[--] $\Gamma$ is $\sem{\_}$-coherent.

  \item[--] For each sentence $\hornsen{X}{H}{(t \trans t')}$ in
    $\Gamma$, $H$ is a finite conjunction of equations of the form
    $t_i = c_i$ where $c_i$ is a constant such that $(c_i)_{E_0}$ is a
    normal form. 

  \item[--] For each sort $s$ of an equation $t_i = c_i$ that occurs
    in a condition of a transition in $\Gamma$ (like above):
  \[
   \begin{array}{lr}
   0_{E_0} \models \lc{\forall{x} \bigor \{ x = c \mid c\in F_{\to s}, \
     c_{E_0} = \sem{c_{E_0}} \}} &
     \mbox{(no junk)} \\[3pt]
     0_{E_0} \models \lc{\bigand} \{ c \neq c' \mid c \neq c' \in F_{\to
     s}, \ c_{E_0} = \sem{c_{E_0}}, c'_{E_0} = \sem{c'_{E_0}}  \} &
     \mbox{(no confusion)}  
   \end{array}                           
  \]
  
  \end{itemize}
  then $h(\rew_{\Gamma/E}) \ \subseteq \ \rew_{\Gamma/\sem{\_}}$. 
\end{corollary}

\begin{proof}
In Corollary \ref{comp-poa-cor}. we set $B = 0_{E_0}, B' = 0_E,
N=N_{E_1 / E_0}, \beta = \sem{\_}$. 
\end{proof}

Corollaries \ref{sound-poa-comp} and \ref{comp-poa-comp} together give
the desired isomorphism \eqref{poa-comp}. 

\subsubsection*{On the conditions underlying completeness}
The situation of the specific conditions of Proposition
\ref{comp-poa-prop}, also 
propagated as conditions of Corollaries \ref{comp-poa-cor} and
\ref{comp-poa-comp}, remind us of the so-called \emph{minimax}
principle present in many areas of mathematics.
From the mathematical side they appear quite restrictive, but from
the applications side they appear quite permissive.
Their great applicability comes from the treatment of conditions as
Boolean terms.
This is the only treatment of conditions available in CafeOBJ, and is
available also in Maude.\footnote{In both cases this was inherited
  from OBJ3.}
As already mentioned this approach has much greater expressivity power
than the standard approach to conditions as conjunctions of atoms.
In that approach a condition $H$ is an equation of the Boolean sort of
the form $\tau = \top$ (true) or $\tau = \bot$ (false), where $\tau$
is a term of the Boolean sort.
Then the `no junk' condition says that the interpretations of $\top$
and $\bot$ are the only elements of $0_E$ and the `no confusion'
condition say that these interpretations are different. 
In fact because the Boolean data type is always imported in
`protected' mode, for all correct programs this situation holds in all
models.  
That $\top$ and $\bot$ are in normal form is trivial because we never
place them in the lefthand sides of equations that we write in our
programs, in fact in the programs we apriori think of them as playing
roles of normal forms.
The following example shows how this works in practice.

\begin{example}
  Let us consider the benchmark example of \texttt{BUBBLE-SORT}.
  The condition of the transition is \texttt{(n < m) = true} which is
  indeed of the form $\tau = \top$.
  Moreover in all \texttt{BUBBLE-SORTS} variants discussed in this
  paper the data type of the Booleans is of course `protected' by the
  `no junk' and `no confusion' conditions. 
\end{example}

There is the question regarding the apparent absence of the $\htr$
component from the conditions. 
In the Boolean terms approach the atomic transitions can be present in
a coded format by employing for each system sort a predicate for the
reachability relation.
For instance CafeOBJ implements these; they are denoted \texttt{==>}.
But Maude does not do this, which means that its approach to
conditions is weaker.

In general the theory around Maude does not consider transitions
conditioned by Boolean terms under the model theoretic constraint that
the data type of the Booleans is protected.
From that literature the closest to this comes the ``generalised
rules'' of \cite{meseguer2020} that allows conditions to be arbitrary
quantifier-free sentences. 

With respect to the coherence conditions, the literature around Maude
puts forward  conditions for the completeness of the reachability
computations that bear some similarity to our $\beta$-coherence, but
they play a different role than in our work, of a mere syntactic
nature.
While in our approach the $\beta$-coherence condition represents a
cause for the existence of the computational model as preordered
algebra, in \cite{meseguer2020} this is not the case due to the
primitive structure of their models.

%%%%%%%%%%%%%%%%%%%%%%%%%%%%%%%%%%%%%%%%%%%%%%
\section{Compositionality of algebraic rewriting}
\label{comp-sec}

Large programs are built efficiently from smaller ones by using
modularisation techniques such as module imports or parameterised
modules.
One of the established mathematical foundations for these is the
pushout technique
\cite{bg80,ins,modalg,iimt,AlgStrucSpec,sannella-tarlecki-book}, etc.
In this section we develop a result on the compositionality of the
algebraic rewriting relation in the context of the pushout technique for
modularisation.
This modularisation technique is by far best defined and developed
within the institution theoretic framework.
It relies on two rather technical concepts: model amalgamation and
theory morphisms.
Below we recall them briefly without directly involving any
institution theory.
Unless specified explicitly all definitions and results below apply to
both \MSA\ and \POA. 

\subsubsection*{Model amalgamation}
A commutative square of signature morphisms like below:
\[
  \xymatrix{
  \Sigma \ar[r]^{\varphi_1} \ar[d]_{\varphi_2} & \Sigma_1
  \ar[d]^{\theta_1} \\
  \Sigma_2 \ar[r]_{\theta_2} & \Sigma'
  }
\]
is a \emph{model amalgamation} square when
\begin{itemize}

\item for any $\Sigma_k$-models $M_k$, $k=\ol{1,2}$, such that
  $\varphi_1 M_1 = \varphi_2 M_2$ there exists an unique
  $\Sigma'$-model $M'$ such that $\theta_k M' = M_k$, $k=1,2$, and 

\item the same property like above when `model' is replaced by `model
  homomorphism'. 

\end{itemize}
The \MSA\ version of the following result appears in many places in
the algebraic specification literature, but its \POA\ version is less
known because \POA\ is less known.
However the \POA\ version can also be found in some places such as
\cite{iimt}.\footnote{For the abridged version of the \POA.}

\begin{proposition}\label{push-amg}
Any pushout square of signature morphisms is a model amalgamation
square. 
\end{proposition}

The concept of model amalgamation has many variations in the
literature.
Sometimes it considers only the condition on the models and drops that
on homomorphisms, other times it drops the uniqueness requirement.
Model amalgamation is a property that holds naturally in many
contexts.
For instance the model $M'$ just puts together the interpretations of
sorts, operations, preorders, provided by $M_1$ and $M_2$.
In order for this to be possible it is necessary that the two
collections of interpretations are mutually consistent, i.e. they
share the same interpretations of the common symbols. 

\subsubsection*{Theory morphisms}
A \emph{theory} is a pair $(\Sigma,E)$ that consists of a signature
$\Sigma$ and a set $E$ of $\Sigma$-sentences.
A \emph{$(\Sigma,E)$-model} is a $\Sigma$-model that satisfies $E$.
A \emph{theory morphism} $\varphi \co (\Sigma,E) \to (\Sigma',E')$ is
a signature morphism $\varphi \co \Sigma\to \Sigma'$ such that for
each $(\Sigma',E')$-model its $\varphi$-reduct is a
$(\Sigma,E)$-model.

\begin{fact}
The composition of theory morphisms yields a theory morphism.
\end{fact}

Pushouts and model amalgamation can be lifted from signature morphisms
to theory morphisms.

\begin{proposition}\label{push-theories}
Any span of theory morphisms has a pushout that inherits a pushout of
the underlying span of signature morphisms. 
\end{proposition}

\begin{proposition}\label{amg-th}
Any pushout square of theory morphisms is a model amalgamation
square. 
\end{proposition}
The last proposition lifts the model amalgamation property from 
signature morphisms to theory morphisms.
This works at an abstract general level, not only for \MSA\ and \POA,
but that generality requires the mathematical machinery of institution
theory.  

\subsubsection*{Amalgamation of rewriting}
The result of this section, developed below, can be read in two ways.
From the perspective of model amalgamation theory it says that
\noindent
\begin{quote}
the amalgamation of algebraic rewriting relations yields an algebraic
rewriting relation.
\end{quote}
In other words, the amalgamation of the free \POA\ models (as given by
Theorem \ref{rew-adjoint}) yields a free \POA\ model.
But from an operational perspective it says that any algebraic
rewriting of the result of putting together two sets of transitions
happens in one of the components, provided that the component
algebraic rewriting relations are mutually consistent. 

\begin{theorem}
Consider a pushout square of \POA\ signature morphisms like in the
diagram below:
\[
  \xymatrix{
  (S,D,F) \rto^{\varphi_1} \dto_{\varphi_2} & (S_1, D_1, F_1)
  \dto^{\theta_1} \\
  (S_2, D_2, F_2) \rto_{\theta_2} & (S', D', F')
  }
\]
Let $B'$ be any \POA\ $(S',D',F')$-algebra and let $\Gamma_1,
\Gamma_2$ be sets of transitions for the signatures $(S_1, D_1, F_1)$
and $(S_2, D_2, F_2)$, respectively.
Let $\Gamma' = \theta_1 \Gamma_1 \cup \theta_2 \Gamma_2$.
If for each $b, b' \in \varphi_1 (\theta_1 B') =  \varphi_2 (\theta_2 B')$
\[
  b \rew_{\Gamma_1/ \theta_1 B'} b' \mbox{ \ if and only if \ }
  b \rew_{\Gamma_2/ \theta_2 B'} b'
\]
then
\[
  \rew_{\Gamma'/B'} \ = \
  \rew_{\Gamma_1 / \theta_1 B'} \ \cup \ \rew_{\Gamma_2 / \theta_2 B'}. 
\]
\end{theorem}

\begin{proof}
We have that
\[
  \xymatrix{
  ((S,D,F),\emptyset) \rto^{\varphi_1} \dto_{\varphi_2} & ((S_1, D_1,
  F_1), \Gamma_1)
  \dto^{\theta_1} \\
  ((S_2, D_2, F_2), \Gamma_2) \rto_{\theta_2} & ((S', D', F'), \Gamma')
  }
\]
is a pushout square in the category of the \POA\ theory morphisms.
Then we can apply Proposition \ref{amg-th} for \POA.
Hence the square above is a model amalgamation square.

Let $B = \varphi_1 (\theta_1 B') =  \varphi_2 (\theta_2 B')$.
Then $(B,\leqslant)$ is a preordered algebra where $\leqslant$ is
defined by $b \leqslant b'$ if and only if $b \rew_{\Gamma_k/ \theta_k
  B'} b'$. 
Moreover
\[
(B,\leqslant) = 
\varphi_1 (\theta_1 B', \rew_{\Gamma_1/ \theta_1 B'}) =
\varphi_2 (\theta_2 B', \rew_{\Gamma_2/ \theta_2 B'}).
\]
By the model amalgamation property for \POA\ theory morphisms
(Proposition \ref{amg-th}) there exists in $(S',D',F')$ an 
amalgamation of $(\theta_1 B', \rew_{\Gamma_1/ \theta_1 B'})$ and
$\varphi_2 (\theta_2 B', \rew_{\Gamma_2/ \theta_2 B'})$.
By the uniqueness of model amalgamation in \MSA\ (Proposition
\ref{push-amg}) it follows that the underlying \MSA\ algebra of the
model amalgamation in \POA\ is $B'$, hence the model amalgamation in
\POA\ gives a preordered algebra $(B',\leqslant)$. 

By the Satisfaction Lemma in \POA\ it follows that $(B',\leqslant)
\models \theta_k \Gamma_k$ hence $(B',\leqslant) \models \Gamma'$. 
Now we prove more, that $(B',\leqslant)$ is the \emph{free} over the
\MSA\ algebra $B'$. 

Let $(A',\leqslant)$ be any \POA\ $(S',D',F')$-algebra such that
$(A',\leqslant) \models \Gamma'$ and let $h' \co B' \to A'$ be an
\MSA\ homomorphism. 
By the Satisfaction Lemma in \POA, for each $k \in \{ 1,2 \}$ we have
that $\theta_k (A',\leqslant) \models \Gamma_k$.
Thus by Theorem \ref{rew-adjoint} it follows that $\theta_k h'$ is a
\POA\ homomorphism
$(\theta_k B',\rew_{\Gamma_k / \theta_k B'}) \to \theta_k (A',\leqslant)$.
By the uniqueness of the model homomorphisms amalgamation property in
\POA\ (Proposition \ref{push-amg}) we obtain that $h'$ is a \POA\
homomorphism $(B',\leqslant) \to (A',\leqslant)$, which proves the
freeness property of $(B',\leqslant)$.
From this freeness it follows that the identity \MSA\ homomorphism  
$1_{B'} \co B' \to B'$ is an isomorphism $(B',\leqslant) \to
(B',\rew_{\Gamma'/B'})$ which implies that $\leqslant \ = \
\rew_{\Gamma'/B'}$.
Thus $(B',\rew_{\Gamma'/B'})$ is the amalgamation of
$(\theta_1 B', \rew_{\Gamma_1 / \theta_1 B'})$ and
$(\theta_2 B', \rew_{\Gamma_2 / \theta_2 B'})$ which implies the
conclusion of the theorem. 
\end{proof}

This was about amalgamation of denotational models of algebraic
rewriting.
What about the amalgamation of computational models in the sense of
the development of Section \ref{comput-sec}?
This is a mathematically interesting problem but without much
applications because in practice we are interested only in
\emph{complete} computational models, and those coincide (modulo
isomorphisms) with the denotational models.  

%%%%%%%%%%%%%%%%%%%%%%%%%%%%%%%%%%%%%%%%%%%%%%%
\section{Conclusions and Future Research}

In this paper we did the following:
\begin{enumerate}

\item We defined in arbitrary \MSA\ algebras $B$ the rewriting relation
  $\rew_{\Gamma/B}$ determined by a set of conditional transitions
  $\Gamma$.

\item We proved that $(B,\rew_{\Gamma/B})$ is the \emph{free
    preordered algebra} that satisfies $\Gamma$.

\item We defined an abstract operational model for non-deterministic
  rewriting-based computations that captures the integration between
  the transition-based and the equational rewriting, both of them
  eventually modulo axioms.
  This operational model provides model theoretic foundations for the
  execution engines of Maude and CafeOBJ.

\item Under a common coherence condition the abstract operational
  model is proper \POA\ models (a preordered algebra) that is sound
  and complete with respect to the denotational preordered algebra
  of the respective \POA\ theory as provided by the adjunction
  result. 

\item The completeness result relies on a set of specific conditions
  that apply well to the conditions-as-Boolean-terms approach of Maude
  and CafeOBJ. 
  In the case of the latter language this is the only approach to
  conditional axioms.

\item Finally, we developed a compositionality result for the
  rewriting relations $\rew_{\Gamma/B}$ within the context of the
  pushout-based technique for modularisation, which may serve as a
  foundation for an institution theoretic comprehensive approach to
  the modularisation of non-deterministic rewriting.  
  
\end{enumerate}
As with other algebraic specification / programming frameworks,
algebraic non-deterministic rewriting can be refined by adding new
features.
Such an example is order-sortedness.
Such enhancements can be topics of further developments in the area.
With our compositionality result of Section \ref{comp-sec} we just
touched the great theme of modularisation.
This can be further developed for algebraic non-deterministic
rewriting by employing the corresponding institution theory applied to
\POA.
% Finally, our foundational work may constitute the basis for defining
% an implementing a concrete programming language that on the one hand
% may resemble a simplified Maude, but on the other hand may depart from
% Maude in the aspects where our theory diverges from rewriting logic as
% has been explained in this paper.

\subsection*{Acknowledgements}
This work was supported by a grant of the Romanian Ministry of
Education and Research, CNCS -- UEFISCDI, project number
PN-III-P4-ID-PCE-2020-0446, within PNCDI III.
The reviewers have made some comments that helped to improve some
aspects of the presentation.

%%%%%%%%%%%%%%%%%%%%%%%%%%%%%%%%%%%%%%%%%%%%%%
\bibliographystyle{plain}
\bibliography{/Users/diacon/TEX/tex}

\end{document}